\documentclass[10pt]{amsart}

\usepackage{amsmath}
\usepackage{amssymb}
\usepackage{bm}
\usepackage{graphicx}
\usepackage{color}
\usepackage{xcolor}
\usepackage{float}
\usepackage{breqn}

\definecolor{amethyst}{HTML}{8b5cf6}
\definecolor{timothy}{HTML}{8b5cf6}

\definecolor{orange}{rgb}{0.988, 0.416, 0.012}
\definecolor{tony}{rgb}{0.988, 0.416, 0.012}

\definecolor{brown}{rgb}{0.588, 0.337, 0.208}
\definecolor{alan}{rgb}{0.588, 0.337, 0.208}

\definecolor{red}{rgb}{1,0,0}
\definecolor{felix}{rgb}{1,0,0}

\usepackage{hyperref}
\hypersetup{colorlinks=true, linkcolor=blue, citecolor=magenta, urlcolor=green}
\usepackage{url}
\usepackage{algpseudocode}
\usepackage{fancyhdr}
\usepackage{mathtools}
\usepackage{tikz-cd}
\usepackage{xy}
\input xy
\xyoption{all}
\usepackage{stmaryrd}
\usepackage{calrsfs}

\newtheorem{theorem}{Theorem}[section]
\newtheorem{lemma}[theorem]{Lemma}
\newtheorem{proposition}[theorem]{Proposition}
\newtheorem{corollary}[theorem]{Corollary}
\theoremstyle{definition}
\newtheorem{definition}[theorem]{Definition}
\newtheorem{example}[theorem]{Example}

\theoremstyle{remark}
\newtheorem{remark}[theorem]{Remark}
\numberwithin{equation}{section}

\newcommand{\aaa}{\mathbb{A}}
\newcommand{\cc}{\mathbb{C}}

\newcommand{\nn}{\mathbb{N}}
\newcommand{\pp}{\mathbb{P}}

\newcommand{\qq}{\mathbb{Q}}
\newcommand{\rr}{\mathbb{R}}
\newcommand{\zz}{\mathbb{Z}}

\providecommand\ldb{\llbracket}
\providecommand\rdb{\rrbracket}

\newcommand{\uu}{\mathcal{U}}

\newcommand{\supp}{\text{supp }}


\begin{document}
	
\mbox{}

\title{On the additive structure of algebraic valuations of polynomial semirings II}

\author{Timothy Chen}
\address{PRIMES-USA\\MIT\\Cambridge, MA 02139}
\email{ctimothy@mit.edu}

\author{Felix Gotti}
\address{Department of Mathematics\\MIT\\Cambridge, MA 02139}
\email{fgotti@mit.edu}

\author{Tony Lu}
\address{PRIMES-USA\\MIT\\Cambridge, MA 02139}
\email{tlu27@mit.edu}

\author{Alan Yao}
\address{PRIMES-USA\\MIT\\Cambridge, MA 02139}
\email{alanyao2008@gmail.com}

\date{\today}

\begin{abstract}
    For $\alpha \in \cc$, let $\mathbb{N}_0[\alpha]$ be the subsemiring of~$\mathbb{C}$ obtained as a homomorphic image of the $\alpha$-evaluation map $\nn_0[x] \to \cc$ defined as $p(x) \mapsto p(\alpha)$ for each polynomial $p(x) \in \nn_0[x]$. Fundamental arithmetic and atomic aspects of the additive structure of $\mathbb{N}_0[\alpha]$ were first studied by the second author and Correa–Morris (2022). In this paper, we continue the investigation, now from the valuation–theoretic perspective. 

    We show that for any algebraic number $\alpha$, the additive monoid of $\mathbb{N}_0[\alpha]$ contains no additive irreducibles if and only if it is isomorphic to the direct product of finitely many isomorphic valuation monoids (monoids whose principal ideals form a chain under inclusion). For any algebraic number $\alpha \in (0,1)$, these valuation monoids are precisely those where $\alpha^{-1}$ is a Perron number having no positive conjugates other than itself. In addition, we offer a description of the algebraic parameters $\alpha$ for which the additive structure of $\mathbb{N}_0[\alpha]$ is a valuation monoid. Finally, we argue that the subset of $(0,1)$ consisting of all algebraic parameters $\alpha$ such that the additive structure of $\mathbb{N}_0[\alpha]$ is a valuation monoid is dense in $(0,1)$.
\end{abstract}

\bigskip
\maketitle

\medskip
\section{Introduction}

Let $\mathbb{N}_0[x]$ be the semiring of all polynomials in an indeterminate $x$ with nonnegative coefficients. The main purpose of this paper is to investigate, for complex parameters $\alpha$, the additive monoid
\[
  M_\alpha := \{p(\alpha) : p(x) \in \mathbb{N}_0[x]\}
\]
arising as the additive structure of the subsemiring $\mathbb{N}_0[\alpha]$ of~$\mathbb{C}$. When $\alpha$ is transcendental over the rationals, $M_\alpha$ is isomorphic to the additive monoid of $\mathbb{N}_0[x]$, which is the free commutative monoid on a countable set. Therefore, we tacitly assume that~$\alpha$ is algebraic throughout this paper. We denote by $\aaa$ the set of all algebraic numbers.
\smallskip

An additive commutative monoid is called atomic if every non-invertible element can be expressed as a finite sum of atoms (i.e., irreducible elements), which means that there are enough atoms in the monoid to create atomic decompositions of any non-invertible element. On the opposite end of the atomic spectrum, a monoid is called antimatter if it has no atoms at all---a term introduced by Coykendall, Dobbs, and Mullins~\cite{CDM99} in the setting of integral domains. It is known that for each $\alpha \in \mathbb{C}$ the monoid $M_\alpha$ is either atomic or antimatter (see~\cite[Theorem~4.2]{CG22}). Since the atomic case was the central focus of that earlier work, the present article concentrates on the complementary class of antimatter monoids. 
\smallskip

\newpage
This paper continues the program initiated in~\cite{CG22}, where the arithmetic and factorization properties of the monoids $M_\alpha$ were examined. Most of the results there relied on the assumption that $M_\alpha$ is atomic, with special emphasis on the subclass of monoids $M_\alpha$ satisfying the ascending chain condition on principal ideals. Here, we undertake a complementary analysis emphasizing those monoids having no atoms at all---that is, the antimatter monoids---and, within this class, we focus on the subclass of valuation monoids. We obtain characterizations of both the antimatter monoids and valuation monoids inside the class consisting of all monoids $M_\alpha$, and these characterizations are in terms of the minimal polynomial of~$\alpha$.
\smallskip

The monoids $M_\alpha$ have drawn increasing attention in recent years. These additive monoids seemed to be first considered in~\cite[Section~5]{GG18}, where the positive rational parameters~$q$ for which $M_q$ is atomic were determined. The special setting where the parameter $q$ of $M_q$ is positive rational was studied deeper by Chapman et al.~\cite{CGG20}, where the authors focus on the study of the length sets of $M_q$ and related factorization invariants, proving that the length set of any nonzero element $r \in M_q$ is an arithmetic progression with common difference $|\mathsf{n}(r) - \mathsf{d}(r)|$. For the rational setting, further factorization invariants and arithmetic properties of $M_q$ were carried out by Albizu-Campos et al.~\cite{ABP21}, who considered not only the monoids $M_q$ but also the submonoids of $M_q$ generated by all the powers $q^n$ whose exponents~$n$ belong to a given numerical monoid. The existence of certain canonical representations inside the rational monoids $M_q$ has been recently studied in \cite{CGGP25} by Chapman et al.
\smallskip

The first general and systematic investigation of the additive monoids $M_\alpha$, where $\alpha$ is taken to be any nonnegative real number, was carried out in~\cite{CG22} by Correa-Morris and the second author. In the same paper, the authors establish several foundational results on the atomicity, factorization, and the structure of principal ideals of~$M_\alpha$, putting special emphasis on the classical factorization properties considered by Anderson et al.~in their landmark paper~\cite{AAZ90}. Motivated by~\cite{CG22}, some other authors have recently made interesting contributions to the study of the arithmetic and atomic structure of the additive monoids $M_\alpha$. For instance, for the same class of monoids, Jiang, Li, and Zhu~\cite{JLZ23} have investigated the omega primality and the elasticity, while Ajran et al.~\cite{ABLST23} have investigated the system of length sets, the sets of Betti elements, and the catenary degree.
\smallskip

In Section~\ref{sec:background}, we introduce the relevant notation, common terminology, and the background needed to follow the rest of the paper. In Section~\ref{sec:algebraic considerations}, we briefly present an algebraic result that will allow us to restrict our attention to the algebraic parameters $\alpha$ whose corresponding minimal polynomials cannot be obtained by composing a polynomial in $\nn_0[x]$ with any of the monomials $x^n$ for $n \ge 2$. In Section~\ref{sec:antimatter}, we further explore the conditions for $M_\alpha$ to be antimatter. Recall that for positive algebraic $\alpha$, the monoid $M_\alpha$ is precisely one of atomic or antimatter~\cite[Theorem~4.2]{CG22}. As non-atomic monoids were not considered in that motivating paper, a significant portion of our paper is dedicated to this case. In Section~\ref{sec:valuation}, we focus on identifying the antimatter monoids that are valuation monoids, or products thereof. We provide two major results in this direction. First, we argue the existence of nontrivial valuation monoids $M_\alpha$ of any given positive rank. We then find several exact characterizations for the class of valuation monoids both in terms of algebraic conditions on $\alpha$ and other divisibility properties of $M_\alpha$, and we also present two examples illustrating the intricacies of the proof of this last result.

\bigskip
\section{Background}
\label{sec:background}

\textbf{General Notation.}
As customary, $\zz$, $\qq$, $\rr$, $\aaa$, and $\cc$ will denote the set of integers, rational numbers, real numbers, algebraic complex numbers, and complex numbers, respectively. We let $\pp$, $\nn$, and $\nn_0$ denote the set of rational primes, positive integers, and nonnegative integers, respectively. For $a,b \in \rr$ with $a \le b$, we let $\ldb a,b \rdb$ denote the set of integers between $a$ and $b$, i.e.,
\[
    \ldb a,b \rdb := \{n \in \zz : a \le n \le b\}.
\]
In addition, for a subset $S$ of $\cc$ and an element $r \in \rr$, we set
\[
	S_{\ge r} := \{s \in S \cap \rr : s \ge r\} \quad \text{ and } \quad S_{> r} := \{s \in S \cap \rr : s > r\}.
\]
For any commutative semiring $S$, we let $S[x]$ denote the commutative semiring consisting of all the polynomials with coefficients in $S$. In particular, $\nn_0[x]$ consists of all polynomials with nonnegative integer coefficients. In the scope of this paper, we find it convenient to set
\[
    x\nn_0[x] + c := \{xf(x) + c : f(x) \in \nn_0[x] \}
\]
for each $c \in \zz$. The set $x\nn_0[x] - 1$ will be especially important in the coming sections. For any $f(x) \in x \nn[x] - 1$ with root $\alpha \in \aaa$, we say that $f(\alpha) + 1$ is an \emph{antimatter decomposition} of $\alpha$.

\medskip
\subsection{Commutative Monoids}

An additively written commutative semigroup $S$ is called \emph{cancellative} if for all $a,b,c \in S$ the equality $a+b = a+c$ implies that $b=c$. Although a monoid is usually defined to be a semigroup with an identity element, in the scope of this paper, the term \emph{monoid} refers to a cancellative and commutative semigroup with an identity element. Let $M$ be an additively written monoid. For any subsets $A$ and $B$ of $M$, we write
\[
	A+B := \{a+b : a \in A \text{ and } b \in B\},
\] 
and if $A = \{a\}$ for some $a \in M$ then we often write $a+B$ instead of $\{a\} + B$. The \emph{group of units} of~$M$ is the abelian group $\uu(M)$ consisting of all invertible elements of~$M$. Two elements $a,b \in M$ are called \emph{associates} if $a \in b + \uu(M)$ (or, equivalently, $b \in a + \uu(M)$). The \emph{reduced monoid} of $M$, denoted $M_{\text{red}}$, is the quotient $M/\uu(M)$. We say that $M$ is \emph{reduced} if $\uu(M)$ is the trivial group, in which case, we can identify $M$ with $M_{\text{red}}$ via the natural homomorphism $m \mapsto m + \uu(M)$ (for all $m \in M$).
\smallskip

A non-invertible element $a \in M$ is called an \emph{atom} if for all $b,c \in M$, the equality $a = b+c$ implies that $\uu(M) \cap \{b,c\}$ is nonempty. We let $\mathcal{A}(M)$ denote the set consisting of all the atoms of $M$. The notion of an antimatter monoid is essential within the scope of this paper.
\begin{definition}[Coykendall, Dobbs, and Mullins;~\cite{CDM99}]
	A monoid is \emph{antimatter} if its set of atoms is empty.
\end{definition}

An element $b \in M$ is called \emph{atomic} if either $b$ is invertible or $b$ can be written as a sum of finitely many atoms of $M$ (allowing repetitions). Following Cohn~\cite{pC68}, we say that the monoid $M$ is \emph{atomic} if every element of $M$ is atomic. We let $\mathsf{Z}(M)$ denote the free commutative monoid on the set $\mathcal{A}(M_{\text{red}})$, and let $\pi \colon \mathsf{Z}(M) \to M_{\text{red}}$ be the only monoid homomorphism fixing the subset $\mathcal{A}(M_{\text{red}})$ of $\mathsf{Z}(M)$. For every element $a \in M$, we set
\[
	\mathsf{Z}(a) = \pi^{-1} (a + \uu(M)).
\]
Note that $M$ is atomic if and only if $\mathsf{Z}(a)$ is nonempty for all $a \in M$. An element $a \in M$ is called \emph{factorial} provided that $\mathsf{Z}(a)$ is a singleton. If every element of $M$ is factorial, then~$M$ is called a \emph{unique factorization monoid} (UFM). 

\smallskip
\subsection{Divisibility and the Valuation Property}

For $a,b \in M$, we say that $b$ \emph{divides} $a$ in $M$ and write $b \mid_M a$ if there exists $c \in M$ such that $a = b+c$. An element $p \in M \setminus \uu(M)$ is \emph{primal} if whenever $p \mid_M a+b$ for some $a,b \in M$, one can write $p = a' + b'$ for some elements $a', b' \in M$ such that $a' \mid_M a$ and $b' \mid_M b$. Then we say that the monoid $M$ is called a \emph{pre-Schreier monoid} or \emph{PS~monoid} if every non-invertible element of $M$ is primal. One can readily show that every UFM is a pre-Schreier monoid.
\smallskip

Let~$S$ be a nonempty subset of $M$. An element $d \in M$ is called a \emph{common divisor} of $S$ if $d \mid_M s$ for all $s \in S$. A common divisor $g \in M$ of $S$ is called a \emph{greatest common divisor} (GCD) of $S$ if any other common divisor of $S$ divides $g$ in~$M$. We denote the set consisting of all GCDs of~$S$ by either $\gcd_M(S)$ or $\gcd(S)$. Observe that any two GCDs of~$S$ in~$M$ are associates. Therefore, if the set consisting of all the GCDs of~$S$ is nonempty, then it must have the form $g + \uu(M)$ for some $g \in M$. If every nonempty finite subset of $M$ has a GCD in $M$, then $M$ is called a \emph{GCD monoid}. It is well known and not difficult to verify that every UFM is a GCD monoid.
\smallskip

The primary property we investigate in this paper is the valuation property, and it can be defined in terms of divisibility in the following way.

\begin{definition}
	A monoid $M$ is a \emph{valuation monoid} if for all $a,b \in M$ either $a \mid_M b$ or $b \mid_M a$.
\end{definition}

Observe that every valuation monoid is a GCD monoid and, therefore, we obtain the following diagram of classes of monoids.
\begin{center}
    \begin{figure}[h]
        \begin{tikzcd} 
            \textbf{ Valuation }     \arrow[r, Leftarrow, shift right=0.8ex] \arrow[red, r, Rightarrow, "/"{anchor=center,sloped}, shift left=0.8ex]
            & \textbf{ GCD }  \arrow[r, Rightarrow, shift right=0.8ex] \arrow[red,  r,  Leftarrow, "/"{anchor=center,sloped}, shift left=0.8ex] 
            & \textbf{ PS } \\
        \end{tikzcd}
        \caption{The (red) marked arrows emphasize that none of the shown implications is reversible.}
        \label{fig:three weaker notions of the FF property}
    \end{figure}
\end{center}

Observe that in the additive monoid $\nn_0$, the divisibility relation coincides with the standard order relation, whence $\nn_0$ is a valuation monoid. In Section~\ref{sec:valuation}, we provide sufficient conditions for a monoid $M_\alpha$ to be a valuation monoid. Now we look at the class consisting of all monoids $M_q$ induced by rational parameters~$q$, and we verify that the three properties in Figure~\ref{fig:three weaker notions of the FF property} are equivalent for monoids in such a class.


\begin{proposition}
    For any $q \in \qq_{>0}$, the following conditions are equivalent.
    \begin{enumerate}
        \item[(a)] $q \in \nn \cup \nn^{-1}$.
        \smallskip

        \item[(b)] $M_q$ is a valuation monoid.
        \smallskip
        
        \item[(c)] $M_q$ is a GCD monoid.
        \smallskip
        
        \item[(d)] $M_q$ is a pre-Schreier monoid.
    \end{enumerate}
\end{proposition}

\begin{proof}
    (a) $\Rightarrow$ (b): If $q \in \nn$, then $M_q = \nn_0$, which is clearly a valuation monoid. If $q \in \nn^{-1}$, then $q = \frac1d$ for some $d \in \nn_{\ge 2}$, and so
    \[
    	M_q = \bigg\langle \frac1{d^k} : k \in \nn \bigg\rangle = \bigg\{ \frac{n}{d^k} : n,k \in \nn_0 \bigg\} = \mathbb{\zz}\bigg[ \frac1{d}\bigg]_{\ge 0},
    \]
    As $\zz\big[\frac1d\big]_{\ge 0}$ is the nonnegative cone of the additive abelian group $\zz\big[\frac{1}d\big]$, the divisibility relation in $M_q$ coincides with the standard order relation. Hence we conclude that $M_q$ is a valuation monoid.
	\smallskip
	
    (b) $\Rightarrow$ (c) $\Rightarrow$ (d): These two implications hold for general commutative monoids.
    \smallskip

    (d) $\Rightarrow$ (a): Assume that the monoid $M_q$ is a pre-Schreier monoid. If $M_q$ is antimatter, then it follows from \cite{GG18} that $q=\frac1d$ for some $d \in \nn_{\ge 2}$. 
	Now assume that $M_q$ is not antimatter. In this case, $M_q$ must be atomic. As $M_q$ is a pre-Schreier, every atom of $M_q$ is also a primal element and so a prime element. Hence $M_q$ is generated by primes, which means that it is a UFM. Now it follows from \cite[Section~6]{GG18} that either $M_q = \nn_0$ or
	\[
		\mathcal{A}(M_q) = \{q^n : n \in \nn_0\}.
	\]
	However, notice that were $q$ an atom of $M_q$, the element $\mathsf{n}(q)$ would have at least two factorizations, namely, $\mathsf{d}(q)$ copies of $q$ or $\mathsf{n}(q)$ copies of $1$---where $\mathsf{n}(q)$ is the numerator of $q$ and $\mathsf{d}(q)$ its denominator. This contradicts $M_q$ being a UFM. Hence $M_q = \nn_0$, which implies that $q \in \nn$.
\end{proof}

Let $R$ be an integral domain. We let $R^*$ and $R^\times$ denote the multiplicative monoid of $R$ and the group of units of $R$, respectively. It is clear that $R^\times = \uu(R^*)$. We say that~$R$ is a \emph{GCD domain} if the multiplicative monoid $R^*$ is a GCD monoid. Assume now that $R$ is a GCD domain. For a nonempty subset $S$ of $R$ not containing $0$, we also refer to any GCD of $S$ in $R^*$ as a GCD of $S$ in $R$.

\medskip
\subsection{Polynomials}

Throughout this section, we let $R$ be an integral domain. For $c_0, \dots, c_d \in R$ such that $c_d \neq 0$, consider the polynomial
\begin{equation} \label{eq:polynomial f(x)}
	f(x) := \sum_{n=0}^d c_n x^n \in R[x].
\end{equation}
For each $i \in \ldb 0,d \rdb$, it is often convenient to denote the coefficient $c_i$ by $[x^i]f(x)$. The \emph{support} of the polynomial $f(x)$ is the set of degrees of its nonzero terms:
\[
    \supp f(x) := \{k \in \ldb 0,d \rdb : c_k \neq 0\}.
\]

\smallskip
Now assume that $R$ is a GCD domain. The \emph{content} of $f(x)$ is the set $c(f) := \gcd(c_0, \dots, c_d)$. If $c(f) = R^\times$, then $f(x)$ is called \emph{primitive}. Gauss's lemma, which we use often throughout this paper, states that the product of primitive polynomials over a GCD domain is primitive. If $r \in c(f)$, then $f(x)/r$ is called a \emph{primitive part} of $f(x)$. When $R = \qq$, there exists unique $r \in \qq_{>0}$ and $p(x), q(x) \in \nn_0[x]$ such that $r f(x)$ is a primitive polynomial in $\zz[x]$, $rf(x) = p(x) - q(x)$, and $\supp p(x)$ is disjoint from $\supp q(x)$. In this case, we call $(p(x), q(x))$ the \emph{minimal pair} of $f(x)$.
\smallskip

We often denote the minimal polynomial of an algebraic number $\alpha$ by $m_\alpha(x) \in \qq[x]$. The \emph{degree} of $\alpha$ is $\deg m_\alpha(x)$ while the \emph{conjugates} of $\alpha$ are the roots of $m_\alpha(x)$. We denote the minimal pair of $m_\alpha(x)$ by $(p_\alpha(x), q_\alpha(x))$, also calling the latter the \emph{minimal pair} of~$\alpha$. The \emph{reciprocal polynomial} of $f(x)$ is the polynomial of $R[x]$ obtained by reversing the coefficients of $f(x)$, that is, $\sum_{n=0}^d c_{d-n}x^n = x^d f(x^{-1})$. For an algebraic number $\alpha$, let $r_\alpha(x)$ denote the reciprocal polynomial of $m_\alpha(x)$:
\begin{equation} \label{eq:reciprocal polynomial of alpha}
	 r_\alpha(x) = x^d m_\alpha\Big(\frac1x \Big).
\end{equation}
\smallskip

We conclude this subsection by recalling Descartes' rule of signs as it will be a helpful tool at our disposal throughout this paper. Assume now that $R=\rr$, and let $f(x)$ be defined as in~\eqref{eq:polynomial f(x)}. We say that $f(x)$ has a \emph{sign variation} at $i \in \ldb 1,d \rdb$ provided that $c_i c_{i-1} < 0$. 

\begin{theorem}[Descartes' rule of signs]
	 The number of sign variations of a nonzero polynomial $f(x) \in \rr[x]$ has the same parity as and is at least the number of positive roots of $f(x)$ (counting multiplicity).
\end{theorem}
\smallskip

Let $\text{Int}(\zz)$ be the ring of integer-valued polynomials, which is the subring of $\qq[x]$ consisting of all polynomials $f(x) \in \qq[x]$ with $f(\zz) \subseteq \zz$. Note that $\zz[x] \subseteq \text{Int}(\zz) \subseteq \qq[x]$. In general, the inclusion $\zz[x] \subseteq \text{Int}(\zz)$ is strict: for instance, $\binom{x}{2} \in \text{Int}(\zz)$ even though it does not belong to $\zz[x]$. In addition, for every $n \in \nn_0$,
\[
	\binom{x}{n} := \frac{x(x-1) \cdots (x - (n-1))}{n!} \in \text{Int}(\zz),
\]
where we assume the convention that $\binom{x}{0} = 1$. The ring $\text{Int}(\zz)$ is a free $\zz$-module with regular basis $\big\{ \binom{x}{n} : n \in \nn_0 \big\}$. Indeed, if we set $\Delta f(k) = f(k+1) - f(k)$, then the Gregory-Newton formula allows us to write any polynomial $f(x)$ of degree $d$ in $\text{Int}(\zz)$ as a unique $\zz$-linear combination of the $\binom{x}{n}$'s as follows:
\begin{equation*} \label{eq:Gregory-Newton formula}
	f(x) = \sum_{j=0}^d \Delta^j f(0) \binom{x}{j}.
\end{equation*}

\medskip
\subsection{Linear Homogeneous Recurrence Relations}

Several of our proofs involve linear recurrence relations. Given a field $F$, a \emph{linear homogeneous recurrence relation of degree} $k$ in $F$ is an equation in countably many variables $(x_n)_{n \ge 0}$ that defines the $n$-th term of a sequence as a linear combination of the previous $k$ terms as follows:
\begin{equation} \label{eq:LHRR}
	x_n = \sum_{j=1}^k c_j x_{n-j}, 
\end{equation}
where $c_1, \dots, c_k \in F$ and $c_{k} \ne 0$. A solution of~\eqref{eq:LHRR} is a sequence $(s_n)_{n \ge 0}$ with terms in $F$ that satisfies~\eqref{eq:LHRR}. The \emph{characteristic polynomial} of the recurrence relation in~\eqref{eq:LHRR} is
\[
	p(x) = x^k - \sum_{j=1}^k c_j x^{k-j}.
\]
It is well known that this type of recurrence relation can be solved explicitly in terms of the roots of their corresponding characteristic polynomials as follows.

\begin{theorem} \label{thm:LHRR solutions} 
    Let $F$ be a field, and let $p(x)$ be a polynomial in $F[x]$ of degree~$d$ that splits as $p(x) = \prod_{i=1}^r (x - \rho_i)^{e_i}$ in the splitting field $K$ of $p(x)$. The set of solutions of the linear recurrence relation with characteristic polynomial $p(x)$ is the $d$-dimensional vector space $V$ over~$F$ with basis
    \[
    	\mathcal{B}_{p(x)} = \big\{ (n^j \rho_i^n)_{n \ge 0} : i \in \ldb 1,r \rdb \text{ and } j \in \ldb 0,e_i - 1 \rdb \big\}.
    \]
    Thus, the vector space $V$ consists of all sequences $(s_n)_{n \ge 0}$ with terms in~$K$ for which there exist polynomials $p_1(x), \dots, p_r(x) \in K[x]$ with $\deg p_i(x) < e_i$ such that
    \[
    	s_n = p_1(n)\rho_1^n + p_2(n)\rho_2^n + \dots + p_r(n) \rho_r^n
    \]
    for every $n \in \nn$.
\end{theorem}

\bigskip
\section{Algebraic Considerations}
\label{sec:algebraic considerations}

In this first section of content, we discuss some preliminary algebraic and divisibility properties of the additive monoids under investigation.

\begin{definition}
    We say that a nonconstant polynomial $f(x) \in \qq[x]$ is \emph{simple} if the only pair $(n,g(x)) \in \nn \times \qq[x]$ that satisfies the equality $f(x) = g(x^n)$ is $(1,f(x))$.
\end{definition}

That is, the GCD of the support of a simple polynomial must be $1$ (in~$\nn$). Not every irreducible polynomial in $\qq[x]$ is simple, as we see in the next example. 

\begin{example}
    For instance, as an immediate application of Eisenstein's criterion, we obtain that the polynomial $m(x) := x^d - q \in \qq[x]$ is irreducible for all pairs $(d,q) \in \nn \times \qq_{>0}$ such that the positive $d\textsuperscript{th}$ root of~$q$ is not rational. Observe that $m(x)$ is simple if and only if $d=1$, whence $x^d - p \in \qq[x]$ is a non-simple irreducible polynomial for any pair $(d,p) \in \nn_{\ge 2} \times \pp$. \hfill $\blacksquare$
\end{example}

The following lemma, which will be helpful in the proof of Theorem~\ref{thm:simple-antimatter-is-valuation}, shows that we can arbitrarily increase the length of our antimatter decomposition while keeping the simplicity condition.

\begin{lemma}\label{lemma:big-and-simple-antimatter-BAGA}
    If $\alpha \in \aaa \cap (0,1)$ is a root of a simple polynomial in $x\nn_0[x]-1$ then, for any given $\ell \in \nn$, the parameter $\alpha$ is also a root of a simple polynomial in $x\nn_0[x] - 1$ whose degree is at least $\ell$.
\end{lemma}

\begin{proof}
    Take $\alpha \in \aaa \cap (0,1)$, and let $p(x)$ be a simple polynomial in $x\nn_0[x] - 1$ having $\alpha$ as a root. The existence, for each $\ell \in \nn$, of a simple polynomial of degree at least $\ell$ in $x\nn_0[x] - 1$ having $\alpha$ as a root reduces to proving the following claim.
    \smallskip
    
    \noindent \textsc{Claim.} There exists a sequence $(p_n(x))_{n \ge 0}$ of simple polynomials in $x\nn_0[x] - 1$ that share $\alpha$ as a common root and satisfy $\deg p_n(x) = 2 \deg p_{n-1}(x)$ for every $n \in \nn$.
    \smallskip

    \noindent \textsc{Proof of Claim.} We proceed by induction. For the base case set $p_0(x) := p(x)$ and to argue the induction step, we assume that we have already produced simple polynomials $p_0(x), p_1(x), \dots, p_n(x)$ in $x\nn_0[x] - 1$ with $\alpha$ as a common root such that $\deg p_i(x) = 2 \deg p_{i-1}(x)$ for every $i \in \ldb 1,n \rdb$. Now consider the polynomial
    \[
        p_{n+1}(x) := p_n(x)(1+x^{\deg p_n(x)}),
    \]
    which also has $\alpha$ as a root and belongs to $x \nn_0[x] - 1$. We only need to argue that $1$ is the only positive common divisor of $\supp p_{n+1}(x)$. Note that every element in $\supp p_n(x)$, perhaps save for $\deg p_n(x)$, remains in $\supp p_{n+1}(x)$. This, along with the fact that $2 \deg p_n(x) \in \supp p_{n+1}(x)$, ensures that the only potential common divisors of $\supp p_{n+1}(x)$ in $\nn$ are $1$ and $2$. Thus, we are done once we show that~$2$ is not a common divisor of $\supp p_{n+1}(x)$.

    Assume, towards a contradiction, that each integer in $\supp p_{n+1}(x)$ is even. As $p_n(x)$ is not a monomial, $\supp p_n(x) \setminus \{ \deg p_n(x) \}$ is a nonempty subset of $\supp p_{n+1}(x)$ and, as a consequence, $\supp p_n(x) \setminus \{ \deg p_n(x) \} \subset 2\nn_0$. Thus, the fact that the only positive common divisor of $\supp p_n(x)$ is $1$ guarantees that $\deg p_n(x)$ is odd. As $p_n(x)$ is simple and non-linear, it cannot be a binomial. Hence $\supp p_n(x) \setminus \{0, \deg p_n(x)\}$ is nonempty, and we can pick $s \in \supp p_n(x) \setminus \{0, \deg p_n(x)\}$. As $s$ is even, $s + \deg p_n(x)$ must be an odd integer in $\supp p_{n+1}(x)$, which is a contradiction.
\end{proof}

Next we prove that for any $\alpha \in \aaa$, the monoid $M_\alpha$ is isomorphic to the direct product of finitely many copies of the monoid $M_\rho$ for some $\rho \in \aaa$ whose minimal polynomial is simple.

\begin{proposition}\label{prop:simple-product}
    For $\alpha \in \aaa$ with minimal polynomial $m_\alpha(x) \in \qq[x]$, if $m_{\alpha}(x) = m(x^k)$ for some $k \in \nn$ and a simple polynomial $m(x) \in \mathbb{Q}[x]$, then the monoids $M_\alpha$ and $M_{\alpha^k}^k$ are isomorphic.
\end{proposition}

\begin{proof}
    Fix $\alpha \in \aaa$ with minimal polynomial $m_\alpha(x)$, and assume that $m_\alpha(x) = m(x^k)$ for a pair $(k, m(x)) \in \nn \times \qq[x]$ such that $m(x)$ is simple. From the fact that $m_\alpha(x)$ is an irreducible polynomial in $\qq[x]$, one obtains that $m(x)$ is also an irreducible polynomial in~$\qq[x]$. Therefore, $m(x)$ is the minimal polynomial of~$\alpha^k$. Before proceeding, it is convenient to argue the following.
    \smallskip

    \noindent \textsc{Claim.} Let $A(x)$ and $B(x)$ be polynomials in $\nn_0[x]$ not both constant, and let $d$ be the maximum of the set $\{\deg A(x), \deg B(x)\}$. Write
    \[
        A(x) = \sum_{n=0}^d a_n x^n \quad \text{ and } \quad B(x) = \sum_{n=0}^d b_n x^n
    \]
    for some coefficients $a_0, \dots, a_d$ and $b_0, \dots, b_d$ in $\nn_0$. If $d = qk + r$ for some $q,r \in \nn_0$ with $r \in \ldb 0, k-1 \rdb$, then the following two conditions are equivalent:
    \begin{itemize}
        \item $A(\alpha) = B(\alpha)$;
        \smallskip

        \item $ \sum_{j=0}^q a_{jk+i} \alpha^{jk} = \sum_{j=0}^q b_{jk+i} \alpha^{jk}$ for every $i \in \ldb 0, k-1 \rdb$.
    \end{itemize}
    \smallskip

    \noindent \textsc{Proof of Claim.} We can assume, without loss of generality, that $\deg A(x) \ge \deg B(x)$, in which case, $d = \deg A(x)$. Then
    \[
        A(\alpha) - B(\alpha) = \sum_{n=0}^d (a_n - b_n)\alpha^n = \sum_{j=0}^q \sum_{i=0}^{k-1} (a_{jk+i} - b_{jk+i})\alpha^{jk+i} =  \sum_{i=0}^{k-1} \bigg(\alpha^i \sum_{j=0}^q (a_{jk+i} - b_{jk+i})\alpha^{jk} \bigg),
    \]
    where $a_{qk+i} = b_{qk+i} = 0$ for every index $i \in \ldb r+1, k-1 \rdb$. As $1, \alpha, \dots, \alpha^{k-1}$ are linearly independent over $\qq(\alpha^k)$, we obtain that $A(\alpha) = B(\alpha)$ if and only if $\sum_{j=0}^q(a_{jk+i} - b_{jk+i}) \alpha^{jk} = 0$ for every $i \in \ldb 0, k-1 \rdb$, which is equivalent to the second condition. This establishes the claim.
    \smallskip

    Let us now continue with the proof of the main statement by defining a function $\psi \colon M_\alpha \to M_{\alpha^k}^k$ as follows: for any polynomial $A(x) \in \nn_0[x]$ having degree $d \in \nn_0$ and $a_0, \dots a_d \in \nn_0$ such that
    \begin{equation} \label{eq:polynomial A(x) temporal}
        A(x) = \sum_{n=0}^d a_n x^n,
    \end{equation}
    write $d = qk + r$ for some $q,r \in \nn_0$ with $r \in \ldb 0, k-1 \rdb$, and then set
    \begin{equation} \label{eq:function psi}
        \psi\big(A(\alpha)\big) := \bigg( \sum_{j=0}^q a_{jk} \alpha^{jk}, \alpha  \sum_{j=0}^q a_{jk+1} \alpha^{jk}, \dots, \alpha^{k-1}  \sum_{j=0}^q a_{jk+(k-1)} \alpha^{jk} \bigg)
    \end{equation}
    so that $a_{qk+i} := 0$ for every index $i \in \ldb r+1, k-1 \rdb$ (note that this index set is empty when $r = k-1$). As an immediate consequence of the established claim, for any two given polynomials $A(x)$ and $B(x)$ in $\nn_0[x]$, the equality $A(\alpha) = B(\alpha)$ guarantees that $\psi(A(\alpha)) = \psi(B(\alpha))$, whence $\psi$ is a well-defined function. In addition, it is clear that $\psi$ is a surjective monoid homomorphism. Finally, for any two polynomials $A(x)$ and $B(x)$ in $\nn_0[x]$, the equality $\psi\big(A(\alpha)\big) = \psi\big(B(\alpha)\big)$ is precisely the second condition in the statement of the established claim, and so the equality $A(\alpha) = B(\alpha)$ must hold. Hence $\psi$ is a monoid isomorphism and, therefore, $M_\alpha\cong M_{\alpha^k}^k$.
\end{proof}

In light of~\cite[Example 3.3]{CG22}, for each positive rational $q \in \qq$, the monoid $M_q$ is a valuation monoid if and only if $q \in \nn \cup \nn^{-1}$. Therefore, for each prime $p \in \pp$, 
\begin{equation} \label{eq:valuation p-adic PMs}
    M_{1/p} = \nn_0\Big[ \frac1p \Big] = \Big\{ \frac{n}{p^k} : k,n \in \nn_0 \Big\}
\end{equation}
is a valuation monoid, and so a GCD monoid.

If the monoid $M_q$ is a GCD monoid, then it is either a UFM or a valuation monoid. This is not the case for the class consisting of all monoids $M_\alpha$ parameterized by non-rational algebraic~$\alpha$. The following example not only illustrates this fact, but also shows, as a special case of Proposition~\ref{prop:simple-product}, how to write certain rank-$d$ positive monoids as finite products of rank-one monoids. 

\begin{example} \label{ex:the monoids M_{d,p}}
    For $d \in \nn_{\ge 2}$, we argue that there are infinitely many non-isomorphic rank-$d$ GCD monoids~$M_\alpha$ (with $\alpha \in \aaa$) that are neither UFMs nor valuation monoids. Note that, for each prime $p \in \pp$, the polynomial
    \[
        m_{d,p}(x) := x^d - \frac1p
    \]
    is irreducible in $\qq[x]$, which follows as an immediate consequence of Eisenstein's criterion. Thus, $m_{d,p}(x)$ is the minimal polynomial of the positive $d\textsuperscript{th}$ root $\rho_{d,p}$ of $\frac1p$. To ease notation we write $M_{d,p}$ instead of $M_{\rho_{d,p}}$. Observe that the polynomial $m_{d,p}(x)$ is not simple as $m_{d,p}(x) = m(x^d)$, where $m(x) := x - \frac1p \in \qq[x]$. The polynomial $x^d - q \in \qq[x]$ is the minimal polynomial of $\rho_{d,p}$, while the polynomial $x-q \in \qq[x]$ is simple. In light of Proposition~\ref{prop:simple-product}, we obtain that
    \[
        M_{d,p} \cong M_{1/p}^d = \nn_0\Big[\frac1p \Big]^d,
    \]
    and so $M_{d,p}$ is isomorphic to the direct product of~$d$ copies of the valuation Puiseux monoid $M_{1/p}$. As the direct product of finitely many GCD monoids is again a GCD monoid, $M_{d,p}$ remains a GCD monoid. Next, as $1,\rho_{d,p},\rho_{d,p}^2,\dots,\rho_{d,p}^{d-1}$ are linearly independent over $\qq$, none of these elements divide each other, so $M_{d,p}$ is not a valuation monoid. Finally, $M_{d,p}$ is not factorial as it is antimatter but not a group. \hfill $\blacksquare$.
\end{example}

In light of the relation between $M_{d,p}$ and $M_{1/p}$, we employ the following notation. Given a polynomial $f(x) \in \qq[x]$, we refer to the unique simple polynomial $g(x) \in \qq[x]$ such that $f(x) = g(x^n)$ for some $n \in \nn$ as the \emph{simplified polynomial} of $f(x)$. In addition, for each $\alpha \in \aaa$, we say that the monoid $M_\alpha$ is \emph{simple} if the minimal polynomial of~$\alpha$ is simple. The \emph{simplified monoid} of $M_\alpha$ is then the monoid generated by a root of the simplified polynomial of $m_\alpha(x)$. It is often helpful to restrict our attention to simple $M_\alpha$ as Proposition~\ref{prop:simple-product} shows that a monoid is isomorphic to a product consisting of copies of its simplified monoid.

\bigskip
\section{Antimatterness}
\label{sec:antimatter}

Fix $\alpha \in \aaa$. Some necessary conditions for $M_\alpha$ to be antimatter are provided in~\cite[Proposition 4.5]{CG22}. Here, we provide a full characterization of when $M_\alpha$ is antimatter, which is a more delicate matter. Ultimately, this characterization will allow us to describe the algebraic parameters $\alpha$ for which $M_\alpha$ is a valuation monoid. 

When $M_\alpha$ is atomic, the fact that the factorization monoid $\mathsf{Z}(M_\alpha)$ is a free commutative monoid on either the set $\{\alpha^n : n \in \mathbb{N}_0 \}$ or the set $\{\alpha^n : n \in \ldb 0,k \rdb \}$ for some $k \in \nn_0$ allows us to identify each factorization in $\mathsf{Z}(M_\alpha)$ with a polynomial in $\mathbb{N}_0[x]$. This was first observed in~\cite[Remark 4.3]{CG22}, and from now on we shall use this identification throughout this paper without explicit mention.

We begin by showing that antimatterness entails $m_\alpha(x)$ having a positive root that is small relative to its other roots. Specifically, we measure this magnitude by the standard Euclidean norm.

\begin{proposition}\label{prop:antimatter-Perron-number}
    For $\alpha \in \aaa_{>0}$, let $M_\alpha$ be antimatter. Then the following statements hold.
    \begin{enumerate}
        \item $\alpha$ is the only positive root of $m_\alpha(x)$.
        \smallskip

        \item Each complex root of $m_\alpha(x)$ is at least $\alpha$ in norm.
    \end{enumerate}
\end{proposition}

\begin{proof}
    (1) By~\cite[Theorem 4.2]{CG22}, $M_\alpha$ is antimatter if and only if $1$ is not an atom. As a result, there exists $f(x) \in x\mathbb{N}_0[x]-1$ having $\alpha$ as a root, which represents our antimatter decomposition. In particular, $f(x)+1$ can be identified with a factorization of~$1$ whose terms consist only of nonconstant powers of~$x$. Since $f(x)$ has precisely one variation in sign, Descartes' Rule of Signs ensures that it has one positive root. In addition, the rule asserts that $\alpha$ has multiplicity one, i.e., it is a simple root. As each root of $m_\alpha(x)$ is also one of $f(x)$, this also holds for $m_\alpha(x)$.
    \smallskip

    (2) Let us now consider the negative reciprocal polynomial $g(x)$ of $f(x)$, namely
    \[
        g(x) := -x^{\deg f(x)}f(x^{-1}).
    \]
    Take $c_0, c_1, \dots, c_{d-1} \in \nn_0$ such that
    \[
        g(x) = x^d - \sum_{i=0}^{d-1} c_ix^i.
    \]
    If $\gamma \in \mathbb{C}$ with $\lvert\gamma\rvert > \alpha^{-1}$ then $g(|\gamma|)>0$. From the fact that
    \[
        |\gamma^n| \ > \sum_{i=0}^{d-1}c_i |\gamma|^i \ \ge \ \left| \sum_{i=0}^{d-1}c_i\gamma^i \right|,
    \]
	we obtain that $\gamma$ cannot be a root of $g(x)$. Hence all roots of $g(x)$, and hence of $r_\alpha(x)$, are at most $\alpha^{-1}$ in norm. Reciprocating yields that all roots of $m_\alpha(x)$ are at least $\alpha$ in norm. Finally, note that this becomes strict when $M_\alpha$ is simple.
\end{proof}

It is known that, for $q \in \mathbb{Q}$, if the monoid $M_q$ is antimatter then $q^{-1}$ is an integer. Our next goal is to generalize this necessary condition for any $\alpha \in \aaa$. This generalization may also be seen as an easily verifiable necessary condition for the monoid $M_\alpha$ to be antimatter. Given $\alpha \in \aaa$ with minimal polynomial $m_\alpha(x)$, recall that $c_\alpha$ is the unique positive integer such that $c_\alpha m_\alpha(x)$ is a primitive integer polynomial. In particular, we set
\[
    w_\alpha(x) := c_\alpha m_\alpha(x) \in \mathbb{Z}[x].
\]

\begin{proposition}\label{prop:antimatter-algebraic-integer}
    For each $\alpha \in \aaa$, if $M_\alpha$ is antimatter then $w_\alpha(0) = -1$ (equivalently, $\alpha^{-1}$ is an algebraic integer).
\end{proposition}

\begin{proof}
    Since $1 \notin \mathcal{A}(M_\alpha)$, \cite[Theorem~4.2]{CG22} guarantees a nonzero polynomial $g(x) \in x\mathbb{N}_0[x]-1$ having $\alpha$ as a root. Hence $g(x)$ is a multiple of the minimal polynomial of $\alpha$, so we may write $g(x) = q(x) m_\alpha(x)$ for some polynomial $q(x) \in \mathbb{Q}[x]$. Thus, $g(x) = q(x) w_\alpha(x) / c_\alpha$, and, as $g(x) \in \mathbb{Z}[x]$, it follows from Gauss's lemma that the content of $q(x)$ is $c_\alpha$. Further, we can write $g(x) = Q(x)w_\alpha(x)$, where $Q(x) := q(x)/c_\alpha$ is a primitive integer polynomial. Hence $w_\alpha(0) \mid g(0) = -1$, which implies that $w_\alpha(0) \in \{-1, 1\}$.

    However, $w_\alpha(0)=1$ would imply that the last term is positive, forcing the number of sign changes to be even---both the leading coefficient and the constant would be positive, so every sign change from positive to negative would be paired with one in the opposite direction. By Descartes' Rule of Signs, $m_\alpha(x)$ would then have an even number of positive roots, which contradicts the uniqueness of $\alpha$ as a positive root. Hence, $w_\alpha(0)=-1$.
\end{proof}



Perron numbers, first introduced in the context of Perron-Frobenius theory of nonnegative matrices and their spectral properties~\cite{dL84}, play a central role in the rest of this section, where we establish characterizations of the antimatter property for the monoids $M_\alpha$.

\begin{definition}
    A \emph{Perron number} is a real algebraic integer greater than $1$ that exceeds each of its algebraic conjugates in norm.
\end{definition}

Whether the monoid $M_\alpha$ is antimatter is closedly connected with the fact that $\alpha^{-1}$ is a Perron number. 
In order to provide a characterization of the antimatter condition for monoids $M_\alpha$, Perron numbers do not fit the bill entirely as they must be strictly greater in norm than their conjugates, while a necessary condition in Proposition~\ref{prop:antimatter-Perron-number} lacks a strict inequality. On the other hand, when $M_\alpha$ is simple, multiple roots of the same maximal modulus cannot exist~\cite[Theorem]{dB94}. As we shall see, if one restricts $\alpha$ to lie on $(0,1)$, we will be able to characterize the simple antimatter monoids $M_\alpha$ by the conditions that $\alpha^{-1}$ is a Perron number having no positive conjugate. Considering the simplified monoid of $M_\alpha$ is reasonable as the property of being antimatter is preserved under products, as in Proposition~\ref{prop:simple-product}.

As the next proposition indicates, for any algebraic $\alpha \in (0,1)$ such that $\alpha^{-1}$ is a Perron number with no positive conjugate aside from itself, there exists $h(x) \in \mathbb{Z}[x]$ so that $1+h(x)w_\alpha(x)\in x\mathbb{N}_0[x]$. Then, after taking $x = \alpha$, we can write $1=f(\alpha)$ for some polynomial $f(x)\in x\mathbb{N}_0[x]$, and the fact that $f(\alpha)$ involves at least two noninvertible elements yields a nontrivial factorization of $1$. 



\begin{proposition} \label{prop:lonely-Perroney-simple-antimatter}
    For $\alpha\in \aaa \cap (0,1)$, if $\alpha^{-1}$ is a Perron number with no positive conjugate aside from itself then there exists a polynomial $h(x) \in \mathbb Z[x]$ such that $h(x)w_\alpha(x) \in x\mathbb N_0[x] - 1$ and is also simple.
\end{proposition}

\begin{proof}
    The conditions on $\alpha^{-1}$ are precisely those specified in~\cite[Theorem 5(i)]{dH92}, and they guarantee that for all sufficiently large $N \in \nn$, multiplying $(x+1)^N$ by the reciprocal polynomial of $w_\alpha(x)$ yields a polynomial with precisely one sign change. Fix some large enough $N \in \nn$, and then let $r(x)$ denote the resulting polynomial and $u(x)$ the reciprocal polynomial of $r(x)$. Then set
    \[
        d := \deg m_\alpha(x) \quad \text{ and } \quad D:=\deg r(x) = \deg u(x) = d+N.
    \]
    Let $(a_n)_{n \ge 0}$ be a sequence of integers, and let us find an index $k \in \nn$ such that
    \[
        h(x) = a_0+a_1x+\cdots+a_kx^k=\sum_{i=0}^ka_ix^i.
    \]
    Set $f(x) := h(x)w_\alpha(x)$. We begin by choosing $a_0$ through $a_{D-1}$ so that the coefficients of $f(x)$ at each of $1,x,\dots,x^{D-1}$ are zero. This can be done directly; since $u(0)=-1$, one may add or subtract copies of $x^iu(x)$ as needed in order to zero the coefficient of $x^i$. For instance, $a_0=1$ in order to set the constant term of $f(x)$ equal to $0$. Moreover, no terms of higher degree in $h(x)$ would affect previous coefficients, meaning a direct algorithm suffices.

    As $\deg r(x) = D$, we can take $c_0, c_1, \dots, c_D \in \mathbb{Z}$ such that $r(x) = \sum_{i=0}^Dc_ix^i$ and then let the coefficients $a_D$ and onward satisfy
    \[
        \sum_{i=0}^D c_{D-i} a_{n-i} = 0
    \]
    for $n \geq D$. Note that $r(x)$ remains monic, so $c_D=-1$. Hence $a_n = \sum_{i=1}^D c_{D-i} a_{n-i}$,
    meaning each term in the sequence is an integer. Observe also that the coefficients of this linear recurrence are chosen so that taking $h(x)=h_n(x)$ would produce some $f(x)$ for which every term with exponent at most $n$ would have a coefficient equal to zero. It may be that at some truncation, not all coefficients of $f(x)$ are nonnegative, but we shall show that they eventually are.

    Using our background on such recurrences, we now find an explicit formula for $a_n$. Let $r_1,\dots,r_d$ be the roots of $r(x)$ other than $-1$, where the root $-1$ has multiplicity~$N$. Therefore our explicit formula is given by
    \[
        a_n=B_1r_1^n+B_2r_2^n+\cdots+B_dr_d^n+C(n)(-1)^n,
    \]
    where $\deg C(n)<N$. Then we can use the fact that the binomial coefficients form a $\zz$-basis for the ring of integer-valued polynomials to write
    \[
        C(x) = C_0\binom{-x}{0}-C_1\binom{-x}{1}+C_2\binom{-x}{2}-\cdots+(-1)^{N-1}C_{N-1}\binom{-x}{N-1}\in\mathbb{Q}[x],
    \]
    the rationale for which will soon be made clear.

    Without loss of generality, let us set $r_1 := \alpha^{-1}$. As the magnitude of $r_1$ strictly exceeds the magnitude of any other root of $r(x)$, the value of $a_n$ will be dominated by $B_1\alpha^{-n}=B_1r_1^n$ for large $n \in \nn$ so long as $B_1 \neq 0$. Thus, proving that later terms are all positive amounts to showing that $B_1\in\mathbb{R}_{>0}$. Using linear algebra, we may find an exact value for $B_1$ in terms of the other roots. First, by applying our recurrence in the backward direction, we obtain the equations \begin{align*}
        a_0=1&=B_1+B_2+\cdots+B_d+C_0,\\
        a_{-1}=0&=\frac{B_1}{r_1}+\frac{B_2}{r_2}+\cdots+\frac{B_d}{r_d}-C_0+C_1,\\
        a_{-2}=0&=\frac{B_1}{r_1^2}+\frac{B_2}{r_2^2}+\cdots+\frac{B_d}{r_d^2}+C_0-2C_1+C_2,\\
        &\ \:\vdots\\[-0.51em]
        a_{-{D-1}}=0&=\frac{B_1}{r_1^{D-1}}+\frac{B_2}{r_2^{D-1}}+\cdots+\frac{B_d}{r_d^{D-1}}+\sum_{n=0}^{N-1}\binom{D-n}{n}(-1)^{-(D-1-n)}C_n,
    \end{align*}
    where for uniformity we use $r_1$ instead of $\alpha^{-1}$. In matrix form, our coefficients correspond to \vspace*{0.34em}

    \resizebox{0.97\textwidth}{!}{$\displaystyle
        \mathbf{D} = \begin{bmatrix}
            1            & 1            & \cdots  & 1            & 1             & 0                  & 0                           & \cdots  & 0 \\
            r_1^{-1}     & r_2^{-1}     & \cdots  & r_d^{-1}     & -1            & 1                  & 0                           & \cdots  & 0 \\
            r_1^{-2}     & r_2^{-2}     & \cdots  & r_d^{-2}     & 1             & -2                 & 1                           & \cdots  & 0 \\
            \vdots       & \vdots       & \ddots & \vdots       & \vdots        & \vdots             & \vdots                      & \ddots & \vdots \\
            r_1^{-(D-1)} & r_2^{-(D-1)} & \cdots  & r_d^{-(D-1)} & (-1)^{-(D-1)} & (D-1)(-1)^{-(D-2)} & \binom{D-2}{2}(-1)^{-(D-3)} & \cdots & \binom{D-N+1}{N-1}(-1)^{-(D-N)}
        \end{bmatrix},
    $} \vspace*{0.34em}
    
    \noindent which yields
    \[
        \displaystyle\mathbf{D}\begin{bmatrix}
            B_1 &
            B_2 &
            \cdots &
            B_n &
            C_0 &
            C_1 &
            C_2 &
            \cdots &
            C_{N-1}
        \end{bmatrix}^\top=\begin{bmatrix}
            1 &
            0 &
            \cdots &
            0
        \end{bmatrix}^\top.
    \]
    It follows from Cramer's rule that $B_1 = \det\mathbf{N}/\det\mathbf{D}$, where $\mathbf{N}$ is the matrix that results from substituting the column vector on the right-hand side of the above equation into the leftmost column of $\mathbf{D}$. 
    Therefore
    \[
        \det\mathbf{D}=\;\prod_{\mathclap{1\leq i<j\leq d}}\;(r_j^{-1}-r_i^{-1})\cdot\prod_{j=1}^d(1+r_j^{-1})^N.
    \]
    As $\det\mathbf{N}$ is obtaind from $\det\mathbf{D}$ in the limit $r_1 \to \infty$ on both sides of the previous identity, we see that \[
        \det\mathbf{N}=\;\prod_{\mathclap{2\leq i<j\leq d}}\;(r_j^{-1}-r_i^{-1})\cdot\prod_{j=2}^dr_j^{-1}\cdot\prod_{i=2}^d(1+r_i^{-1})^N.
    \] 
    We can rewrite $\det\mathbf{D}$ in a similar form.
    \[
        \det\mathbf{D}=\prod_{\mathclap{2\leq i<j\leq d}}\;(r_j^{-1}-r_i^{-1})\cdot\prod_{j=2}^d(r_j^{-1}-r_1^{-1})\cdot\prod_{i=2}^d(1+r_i^{-1})^N\cdot(1+r_1^{-1})^N.
    \]
    Thus, 
    \[
        \frac{\det\mathbf{N}}{\det\mathbf{D}}=\frac{\displaystyle\prod_{j=2}^dr_j^{-1}}{\displaystyle\prod_{j=2}^d(r_j^{-1}-r_1^{-1})\cdot(1+r_1^{-1})^N}.
    \]
    Note that $(1-r_1^{-1})^N$ is positive because $r_1^{-1}=\alpha<1$. Then it suffices to focus on the remaining piece. In addition,
    \[
        \frac{\displaystyle\prod_{j=2}^dr_j^{-1}}{\displaystyle\prod_{j=2}^d(r_j^{-1}-r_1^{-1})}=\prod_{j=2}^d(1-r_j/r_1)^{-1}.
    \]
    If $r_j$ is negative then $1-r_j/r_1$ is positive. For the remaining complex conjugate pairs, observe that $\overline{1-r_j/r_i}=1-\overline{r_j}/r_1$ because $r_1$ is real. Hence the contribution from that pair is the norm of a nonzero complex number, which is also positive. Thus, we conclude that $B_1\in\mathbb{R}_{>0}$.

    As the roots are being exponentiated, the value of $a_n$ will be dominated by $B_1r_1^n$ for sufficiently large $n$ because $r_1=\alpha^{-1}$ is the strictly largest root by norm. We are now in the position to show that $f(x) \in \nn_0[x]$, which would complete the proof. It is only past this point that we make use of the fact that $r(x)$ has precisely one sign change.

    Recall that $h_k(x)$ is the truncation that includes only terms with exponents at most $k$. Letting $b_n$ denote the coefficient of $x^n$ in $f(x)$, we must prove that $b_n\geq0$ when $n=k+i$ for any $i\in\llbracket1,d\rrbracket$, as all other coefficients are zero. In particular, we will actually show that $b_n>0$ for those $d$ values of $n$, in order to demonstrate that $f(x)$ is simple. The cases for $i=1$ and $i>1$ will be treated distinctly. In either case, however, we will demonstrate that there exists some large enough $k$ for which the relevant coefficient is positive, and by choosing $k$ larger than each of those bounds, we will have satisfied the criteria.

    First, $b_{k+1}=a_{k+1}$ by definition. Moreover, there must exist an index $k \in \nn$ large enough for which $a_{k+1}$ is positive. It is clear that $b_n = \sum_{j=0}^{d-i} c_j a_{k-d+i+j}$ for every $n \in \nn$. Now take $p \in \llbracket1,d-1\rrbracket$ satisfying $c_j \geq 0$ for $j<p$ and $c_j \leq 0$ for $j \geq p$ so as to encode the position of the singular sign change in some sense. There may be multiple potential~$p$ if there is a gap in the support of $r(x)$, in which case any such~$p$ would suffice. Clearly, for $i > d-p$, we will be summing only nonnegative terms, as $j \le d-i < d-(d-p)=p$, so each coefficient $c_j \ge 0$---moreover, we may choose $k \in \nn$ sufficiently large so that $a_{k-d+i+j} > 0$.

    We can restrict our attention to $i\leq d-p$, whence $b_n\geq\sum_{j=0}^{d-1}c_ja_{k-d+i+j}$ because the new sum incorporates only nonnegative terms. In addition,
    \[
        \sum_{j=0}^{d-1}c_ja_{k-d+i+j}=\sum_{j=0}^dc_ja_{k-d+i+j}-c_da_{k+i}.
    \]
    The second term on the right-hand side is positive, while the first term on the right-hand side appears similar to an evaluation of $r(x)$. In particular, as $a_n\approx B_1r_1^n$ in a way that will be made rigorous later, then
    \begin{align*}
        \sum_{j=0}^dc_ja_{k-d+i+j} \approx\sum_{j=0}^dc_jB_1r_1^{k-d+i+j} =B_1r_1^{j-d+i}\sum_{j=0}^dc_jr_1^j =B_1r_1^{j-d+i}r(r_1) =0.
    \end{align*}
    Hence, for sufficiently large $k \in \nn$, the summation nears $0$ while the other term of $-c_da_{k+i}$ grows without bound, so $b_n$ does as well. We prove this more carefully by considering the deviations between terms and their asymptotic approximations. Suppose the inequality 
    \[
        1-x<\frac{a_{k-\ell}}{B_1r_1^{k-\ell}}<1+x
    \]
    holds for each $\ell \in \llbracket0, d-1\rrbracket$. Then, the absolute value of the summation is bounded above by $\displaystyle (d+1)a_{k+i}\max_{0\leq j\leq d}c_j(1-(1-x)^d)$, while the other term is $-c_da_{k+i}$. As $\displaystyle (d+1)\max_{0\leq j\leq d}c_j$, and $-c_d$ are positive constants, $-c_da_{k+i}$ must have the larger magnitude for sufficiently small values of~$x$, which is certainly attainable by simply increasing~$k$. In fact, as each of these terms will be positive, our new polynomial will actually be simple. For instance, $b_{k+1}$ and $b_{k+2}$ will both be nonzero, which implies that the greatest common divisor of the support must be $1$. Hence we have found such a simple $f(x)$ satisfying the desired conditions.
\end{proof}

Now we just need to put together Propositions~\ref{prop:antimatter-Perron-number}, \ref{prop:antimatter-algebraic-integer},~and \ref{prop:lonely-Perroney-simple-antimatter} to obtain the main result of this section, which are the following two characterizations of the simple antimatter monoids $M_\alpha$ in terms of the algebraic parameter~$\alpha$.

\begin{theorem}\label{thm:antimatter-exact}
    For any $\alpha \in \aaa \cap (0,1)$ with minimal polynomial $m_\alpha(x)$, the following conditions are equivalent.
    \begin{enumerate}
        \item [(a)] $M_\alpha$ is a simple antimatter monoid.
        \smallskip

        \item [(b)] $\alpha^{-1}$ is a Perron number and has no positive conjugate aside from itself.
        \smallskip

        \item [(c)] $\alpha^{-1}$ is a Perron number and $m_\alpha(x)$ has a simple multiple $p(x)\in x\mathbb{N}_0[x]-1$.
    \end{enumerate}
\end{theorem}

\begin{proof}
    (a) $\Rightarrow$ (b): This implication follows as a result of combining Propositions~\ref{prop:antimatter-Perron-number} and~\ref{prop:antimatter-algebraic-integer}. Specifically, Proposition~\ref{prop:antimatter-Perron-number}(2) provides that $\alpha^{-1}$ is not exceeded by any of its conjugates, while strict inequality follows from~\cite[Theorem]{dB94} because $m_\alpha(x)$ is simple.
    \smallskip
    
    (b) $\Rightarrow$ (c): This one follows from Proposition~\ref{prop:lonely-Perroney-simple-antimatter}.
    \smallskip

    (c) $\Rightarrow$ (a): Only simple polynomials can have Perron numbers as roots, meaning $r_\alpha(x)$ and thus $m_\alpha(x)$ must be simple. Further, $p(x)$ acts as an antimatter decomposition of~$1$, so it cannot be an atom. This suffices to show that $M_\alpha$ has no atoms by~\cite[Theorem~4.2]{CG22}.
\end{proof}

This easily extends to the case where $M_\alpha$ is not simple.

\begin{corollary}\label{cor:antimatter-exact}
    For any $\alpha \in \aaa_{>0}$, the monoid $M_\alpha$ is antimatter if and only if the simplified polynomial of $r_\alpha(x)$ has a Perron number as a root and $\alpha$ has no positive conjugate aside from itself.
\end{corollary}

\bigskip
\section{The Valuation Property}
\label{sec:valuation}

The primary purpose of this section is to study which monoids $M_\alpha$ are valuation monoids. As every valuation monoid is either antimatter or has its reduced monoid isomorphic to $\mathbb{N}_0$, it suffices to restrict our attention to the monoids $M_\alpha$ that are antimatter. We begin by demonstrating a short example of an irrational monoid without the valuation property, as the Puiseux (submonoids of $\mathbb{Q}_{\geq0}$) valuation monoids have been characterized~\cite[Proposition 3.1]{GGT21}.

The proof of Proposition~\ref{prop:lonely-Perroney-simple-antimatter}, the explanation of the following example, and our later proof of Theorem~\ref{thm:simple-antimatter-is-valuation} that establishes an exact characterization of the valuation monoids all rely on homogeneous linear recurrence relations. However, the proof we present below is especially interesting as the recurrence used is a generalized Fibonacci sequence. Further, it involves Pisot numbers, which are a subclass of Perron numbers. Although the below example follows directly from Theorem~\ref{thm:simple-antimatter-is-valuation}, its relation to Fibonacci numbers makes it interesting and worth exploring.

\begin{proposition} \label{prop:irrational-valuation-example}
    For any $d \in \nn$, there exists $\alpha \in \aaa$ such that the monoid $M_\alpha$ is an antimatter valuation monoid of rank~$d$ that is not a group.
\end{proposition}

\begin{proof}
    When $d=1$, then setting $\alpha = \frac{1}{n}$ for any $n\in\nn$ yields a valuation monoid; moreover, when $n\geq2$, the corresponding monoid is antimatter. Hence, we focus on the case where $d \ge 2$. Consider the polynomial $f(x) := -1 + \sum_{i=1}^{d} x^i \in \zz[x]$, which has precisely one sign change. Let $\alpha$ be the unique positive root guaranteed by Descartes' Rule of Signs. It follows from~\cite[Theorem~2]{aB50} that $x^d - \sum_{i=0}^{d-1}x^i \in \qq[x]$ is irreducible, meaning $f(x)$ is the minimal polynomial of $\alpha$. Further, the paper provides that its root $\alpha^{-1}>1$ is a Pisot-Vijayaraghavan number (which we refer to as a Pisot number), i.e., $\alpha^{-1}$ is a Perron number with the further restriction that all of its conjugates are less than $1$ in norm. For each $n \in \nn_0$, after multiplying the equality $1 = \sum_{i=1}^d \alpha^i$ by $\alpha^n$, we obtain that
    \begin{equation} \label{eq:alpha^n}
            \alpha^n = \sum_{i=1}^d \alpha^{n+i}.
    \end{equation}
    Consider the positive monoid $M_\alpha$, which has rank~$d$. The element $1$ is not an atom of $M_\alpha$ because $1 = \sum_{i=1}^d \alpha^i$, whence the monoid $M_\alpha$ is antimatter by virtue of \cite[Theorem~4.2]{CG22}.
    \smallskip

    To argue that $M_\alpha$ is a valuation monoid, fix $w,w' \in M_\alpha$, and let us prove that the principal ideals $w + M_\alpha$ and $w' + M_\alpha$ are comparable under set inclusion. First, notice that for any $\sum_{i=0}^k a_i \alpha^i \in M_\alpha$ with coefficients $a_0, \dots, a_k \in \nn_0$ and a given $\ell \in \nn$ with $\ell \ge k$, repeated applications of~\eqref{eq:alpha^n} allow us to write $\sum_{i=0}^k a_i \alpha^i = \sum_{i=\ell}^{\ell+d-1} b_i \alpha^i$ for some $b_\ell, b_{\ell + 1}, \dots, b_{\ell+d-1}$. Therefore, for each sufficiently large $r \in \nn_0$, we can take coefficients $c_0, \dots, c_{d-1}$ and $c'_0, \dots, c'_{d-1}\in\mathbb{N}_0$ so that
    \begin{equation}
            w = \sum_{i=0}^{d-1} c_i \alpha^{r+i} \quad \text{ and } \quad w' = \sum_{i=0}^{d-1} c'_i \alpha^{r+i}.
    \end{equation}
    Fix such a sufficiently large $r$ and set $G_i := c_i - c'_i$ for every index $i \in \ldb 0, d-1 \rdb$. We split the rest of the proof into the following cases, taking into account the convention that $0$ has the same sign as both positive and negative numbers.
    \smallskip

    \textsc{Case 1:} the nonzero elements of $\{G_0, \dots, G_{d-1}\}$ are all of the same sign. If $\min \{G_0, \dots, G_{d-1} \} \ge 0$, then $w - w' \in M_\alpha$. We can similarly deduce that $w' - w \in M_\alpha$ when $\max \{G_0, \dots, G_{d-1} \} \le 0$.
    \smallskip

    \textsc{Case 2:} not all $G_0, \dots, G_{d-1} \in \zz$ have the same sign, where again we exclude $0$ from consideration. Suppose for the sake of contradiction that for each $s \in \nn_0$, after writing $w-w'$ entirely in terms of the powers of $\alpha$ from $\alpha^{s+r}$ to $\alpha^{s+r+d-1}$, not all of the $d$ coefficients are of the same sign. Let us argue the following claim.
    \smallskip

    \noindent\textsc{Claim.} For infinitely many $s$, the coefficient of the $\alpha^{s+r}$ term and the coefficient of the $\alpha^{s+r+d-1}$ term have opposite signs.
    \smallskip

    \noindent\textsc{Proof of Claim.} Suppose, towards a contradiction, that there exists some $a \in \nn_0$ for which every index $s \in \ldb a, a+d-1 \rdb$ satisfies the condition that the coefficient of the first term $\alpha^{s+r}$ and the coefficient of the last term $\alpha^{s+r+d-1}$ have the same sign. Since the coefficient of $\alpha^{s+r}$ equals the coefficient of $\alpha^{s+r+d}$ from one value of $s$ to the next after applications of~\eqref{eq:alpha^n}, and adding two numbers of the same sign preserves the sign, we know that for each $s \in \ldb a, a+d-1 \rdb$, the coefficients of the terms $\alpha^{a+r+d-1}, \dots, \alpha^{s+r+d-1}$ will have the same sign. Thus, when $s = a+d-1$, all $d$ coefficients will have the same sign, which is a contradiction. \hfill $\square$
    \smallskip

    Let $(F_n)_{n \ge -d}$ denote the Fibonacci sequence of order~$d$, which is defined as follows: $F_n := 0$ for every $n \in \ldb -d,-3 \rdb$, $F_{-2} := -1$, $F_{-1}:=1$, and
    \[
        F_n := \sum_{k=n-d}^{n-1} F_k
    \]
    for every $n \in \nn_0$. Thus, $F_n = 0$ for every $n \in \ldb 0,d-2 \rdb$ while $F_{d-1} = 1$. We can rewrite $w-w'$ using the terms of the sequence $(F_n)_{n \ge -d}$ as follows:
    \begin{equation} \label{Fibonacci's aux}
            w-w' = \sum_{i=0}^{d-1} G_i \alpha^{r+i} = \sum_{i=0}^{d-1} \ \sum_{j=0}^{d-1} \delta_{i,j} G_j \alpha^{r+i} = \sum_{i=0}^{d-1} \ \sum_{j=0}^{d-1} \sum_{k=i-j-1}^{d-j-2} F_k G_j \alpha^{r+i},
    \end{equation}
    with the last equality due to $\llbracket i-j-1,d-j-2\rrbracket$ only containing $-2$ when $i-j \leq -1$, and containing $-1$ when $i-j \leq 0$, the only two possible $k$ in that interval for which $F_k$ is nonzero.

    Note that
    \begin{align*}
        \sum_{i=0}^{d-1} \ \sum_{j=0}^{d-1} \sum_{k=i-j-1}^{d-j-2} F_{s+k} G_j \alpha^{s+r+i}
        &= \sum_{j=0}^{d-1} \sum_{k=i-j-1}^{d-j-2} F_{s+k} G_j \alpha^{s+r} \\
        &\quad + \sum_{i=1}^{d-1} \ \sum_{j=0}^{d-1} \sum_{k=i-j-1}^{d-j-2} F_{s+k} G_j \alpha^{s+r+i} \\
        &= \sum_{i=1}^d \ \sum_{j=0}^{d-1} F_{s+d-j-1} G_j \alpha^{s+r+i} \\
        &\quad + \sum_{i=1}^{d-1} \ \sum_{j=0}^{d-1} \sum_{k=i-j-1}^{d-j-2} F_{s+k} G_j \alpha^{s+r+i} \\
        &= \sum_{i=1}^{d-1} \sum_{j=0}^{d-1} \bigg(F_{s+d-j-1} + \sum_{k=i-j-1}^{d-j-2} F_{s+k} \bigg) G_j \alpha^{s+r+i} \\
        &\quad + \sum_{j=0}^{d-1} F_{s+d-j-1} G_j \alpha^{s+r+d} \\
        &= \sum_{i=1}^{d-1} \ \sum_{j=0}^{d-1} \sum_{k=i-j-1}^{d-j-1} F_{s+k} G_j \alpha^{s+r+i} \\
        &\quad + \sum_{j=0}^{d-1} \sum_{k=d-j-1}^{d-j-1} F_{s+k} G_j \alpha^{s+r+d} \\
        &= \sum_{i=1}^d \ \sum_{j=0}^{d-1} \sum_{k=i-j-1}^{d-j-1} F_{s+k} G_j \alpha^{s+r+i} \\
        &= \sum_{i=0}^{d-1} \ \sum_{j=0}^{d-1} \sum_{k=i-j-1}^{d-j-2} F_{s+k+1} G_j \alpha^{s+r+i+1}.
    \end{align*}
    As a consequence, for every $s \ge 0$, we conclude that
    \[
        w - w' = \sum_{i=0}^{d-1} \sum_{j=0}^{d-1} \sum_{k=i-j-1}^{d-j-2} F_{s+k} G_j \alpha^{s+r+i}.
    \]
    By the claim, the coefficients of the $\alpha^{s+r}$ and $\alpha^{s+r+d-1}$ terms have opposite signs for infinitely many indices $s \in \nn_0$. Suppose without loss of generality that the coefficient of $\alpha^{s+r}$ is positive and the coefficient of $\alpha^{s+r+d-1}$ is negative for infinitely many $s \ge 0$. Thus,
    \[
        \sum_{j=0}^{d-1} \sum_{k=-j-1}^{d-j-2} F_{s+k} G_j = \sum_{j=0}^{d-1} F_{s+d-j-1} G_j > 0,
    \]
    and
    \[
        \sum_{j=0}^{d-1} \sum_{k=d-j-2}^{d-j-2} F_{s+k} G_j = \sum_{j=0}^{d-1} F_{s+d-j-2} G_j < 0.
    \]
    Because not all $G_i$ for $i\in\ldb 0,d-1 \rdb$ have the same sign, suppose $G_\ell$ is positive for some $\ell \in \ldb 0, d-1 \rdb$. The above two inequalities can be rearranged as follows:
    \[
        -\sum_{\substack{0\leq j\leq d-1\\j\neq \ell}} \frac{F_{s+d-j-1}}{F_{s+d-\ell-1}} G_j < G_\ell < -\sum_{\substack{0\leq j\leq d-1\\j\neq \ell}} \frac{F_{s+d-j-2}}{F_{s+d-\ell-2}}G_j,
    \]
    and we set $L_s$ and $R_s$ to be the left and right bounds, respectively. By \cite[Equation~2]{PW23}, the formula for each value in the Fibonacci sequence of order $d$ is
    \[
        F_n = \sum_{i=1}^d \frac1{\displaystyle\prod_{j \ne i} (\phi_i-\phi_j)} \phi_i^n,
    \]
    where $\phi_i$ for $i \in \ldb 1,d\rdb$ are the roots of $x^d - \sum_{i=0}^{d-1} x^i$. Without loss of generality, take $\phi_1 := \alpha^{-1}$, one of the roots of this polynomial. Then, for each $n \geq 0$, the equality $F_n = C \phi_1^n + E_n$ holds,
    where
    \[
        C := \frac1{\displaystyle\prod_{j\neq1} (\phi_1-\phi_j)} \quad \text{ and } \quad E_n := \sum_{i=2}^d \frac{1}{\displaystyle\prod_{j\neq i} (\phi_i-\phi_j)}\phi_i^n.
    \]
    For $i \in \ldb 2,d \rdb$, we know that $\lvert\phi_i\rvert < 1$ since $\phi_1$ is a Pisot number, so the sequence $(E_n)_{n \ge 0}$ tends to $0$ in magnitude. Further, $\lvert E_n\rvert \leq D\rho^n$ for some fixed $D>0$ and $\rho := \max\{ \lvert\phi_i\rvert : i \in \ldb 2,d \rdb \} \in (0,1)$. After setting $\mu := \max(\ldb 0,d-1 \rdb \setminus \{\ell\})$, we see that
    \begin{align*}
        R_s &= -\sum_{\substack{ 0\leq j \leq d-1 \\ j\neq \ell }} \frac{\left(\tfrac{C}{\alpha^{s+d-j-2}}\right) + E_{s+d-j-2}}{\left(\tfrac{C}{\alpha^{s+d-\ell-2}}\right) + E_{s+d-\ell-2}} G_j
        = -\;\sum_{\mathclap{\substack{ 0\leq j\leq d-1 \\ j\neq \ell}}}\; \alpha^{j-\ell} \frac{C + E_{s+d-j-2}\alpha^{s+d-j-2}} {C + E_{s+d-\ell-2}\alpha^{s+d-\ell-2}} G_j \\
        &\le -\;\sum_{\mathclap{\substack{0\leq j\leq d-1 \\ j\neq \ell}}}\; \alpha^{j-\ell} \frac{C + E_{s+d-\mu-2}\alpha^{s+d-\mu-2}} {C + E_{s+d-\ell-2}\alpha^{s+d-\ell-2}} G_j
        = -\;\sum_{\mathclap{\substack{0\leq j\leq d-1 \\ j\neq \ell}}}\; \alpha^{j-\ell}\frac{1 + \gamma_{s+d-\mu-2} }{1 + \gamma_{s+d-\ell-2} } G_j,
    \end{align*}
    where $\gamma_n := E_n\alpha^n/C$. We can bound this value as follows:
    \[
        \lvert\gamma_n\rvert = \frac{|E_{n}|\,\alpha^n}{|C|} \leq \frac{D}{|C|}(\rho\alpha)^n = D'\sigma^n,
    \]
    where $D' := D/\lvert C\rvert > 0$ is a constant and $\sigma := \rho\alpha$. Note that $\sigma \in (0,1)$ when $\alpha < 1$. Thus, there exists $b \in \nn$ such that $\lvert\gamma_n\rvert < \frac{1}{2}$ for all $n \ge b$. For these values of $n$,
    \[
        \frac{1}{1 + \gamma_{n}} =  \sum_{i=0}^\infty (-\gamma_{n})^i = 1 + \theta_n,
    \]
    where $\theta_n := \sum_{i=1}^\infty (-\gamma_{n})^i$, and this is bounded as
    \[
        \lvert\theta_n\rvert = \left|\sum_{i=1}^\infty (-\gamma_{n})^i\right| \leq \lvert\gamma_n\rvert\sum_{i=0}^\infty |\gamma_{n}|^i = \lvert\gamma_n\rvert\frac{1}{1-|\gamma_n|} \leq 2|\gamma_n|.
    \]
    Therefore, for each $s \in \nn_0$ with $s \ge b+1$, we obtain that $s+d-\ell-2 \geq s-1 \geq b$, whence
    \begin{align*}
        R_s &\leq -\;\sum_{\mathclap{\substack{0 \leq j\leq d-1 \\ j \neq \ell}}}\; \alpha^{j-\ell}\frac{1 + \gamma_{s+d-\mu-2} }{1 + \gamma_{s+d-\ell-2} } G_j \\
        &= - \;\sum_{\mathclap{\substack{0\leq j\leq d-1 \\ j\neq \ell}}}\; \alpha^{j-\ell} G_j (1+\gamma_{s+d-\mu-2})(1+\theta_{s+d-\ell-2}) \\
        &= \delta_s - \;\sum_{\mathclap{\substack{0 \le j \le d-1 \\ j\neq \ell}}}\; \alpha^{j-\ell} G_j,
    \end{align*}
    where $\delta_s := -A(\gamma_{s+d-\mu-2} + \theta_{s+d-\ell-2}(1 + \gamma_{s+d-\mu-2}))$. Now set $\displaystyle A := \;\sum_{\mathclap{\substack{0\leq j\leq d-1 \\ j\neq \ell}}}\; \alpha^{j-\ell}G_j$ and observe that
    \begin{align*}
        |\delta_s| &= |A||\gamma_{s+d-\mu-2} + \theta_{s+d-\ell-2}(1 + \gamma_{s+d-\mu-2})| \\
        &\leq |A|(|\gamma_{s+d-\mu-2}| + |\theta_{s+d-\ell-2}|(1 + |\gamma_{s+d-\mu-2}|)) \\
        &\leq |A|(|\gamma_{s+d-\mu-2}| + 2|\gamma_{s+d-\ell-2}|(1 + |\gamma_{s+d-\mu-2}|)) \\
        &\leq |A|(D'\sigma^{s+d-\mu-2} + 2D'\sigma^{s+d-\ell-2}(1 + |\gamma_{s+d-\mu-2}|)) \\
        &= |A|\sigma^{s}(D'\sigma^{d-\mu-2} + 2D'\sigma^{d-\ell-2}(1 + |\gamma_{s+d-\mu-2}|)) \\
        &\leq |A|\sigma^{s}(D'\sigma^{d-\mu-2} + 2D'\sigma^{d-\ell-2}(1 + |\gamma_{d-\mu-2}|)),
    \end{align*}
    with the last line due to the sequence $(\gamma_n)_{n \ge 0}$ being strictly decreasing. As a consequence, after setting $P := |A|(D'\sigma^{d-\mu-2} + 2D'\sigma^{d-\ell-2}(1 + |\gamma_{d-\mu-2}|)) > 0$, the inequality $|\delta_s| \le P\sigma^s$ holds for all sufficiently large $s \in \nn_0$. Similarly, we can set $Q := \lvert A\rvert(D'\sigma^{d-\mu-1} + 2D'\sigma^{d-\ell-1}(1 + |\gamma_{d-\mu-1}|)) > 0$ and argue that
    \[
        L_s = \varepsilon -\;\sum_{\mathclap{\substack{0\leq j\leq d-1 \\ j\neq \ell}}}\; \alpha^{j-\ell} G_j
    \]
    for some $\lvert\varepsilon_s\rvert \leq Q \sigma^s$ that holds for all sufficiently large $s \in \nn_0$. Since $L_s < G_\ell < R_s$ for infinitely many $s$ and $\lvert R_s-L_s\rvert = \lvert \delta_s - \varepsilon_s\rvert \leq (P+Q)\sigma^s < 1$ for all sufficiently large $s \in \nn_0$ and approaches $0$, and given that $G_\ell$ is an integer,
    \[
        G_\ell = -\;\sum_{\mathclap{\substack{0\leq j\leq d-1\\j\neq \ell}}}\; \alpha^{j-\ell} G_j.
    \]
    Therefore, $\sum_{j=0}^{d-1} G_j \alpha^j=0$, so $\alpha$ is the root of a polynomial in $\zz[x]$ with degree $d-1$. However, this contradicts that the minimal polynomial of $\alpha$ has degree $d$, from which we deduce that either $w - w' \in M_\alpha$ or $w'-w  \in M_\alpha$. Hence we conclude that $M_\alpha$ is a valuation monoid of rank~$d$.
\end{proof}

For instance, the case of $d=2$ yields the reciprocal of the golden ratio.

\medskip
\subsection{A Sufficient Condition}

In a similar vein to the above proposition, we employ our understanding of the antimatter condition on $M_\alpha$ to show some simple conditions for $M_\alpha$ to be a valuation monoid. We begin with the following lemma.

\begin{lemma} \label{lem:rational homogeneous linear recurrences}
    Let $(a_j)_{j \ge 1}$ be a sequence of rationals that satisfies a homogeneous linear recurrence governed by the characteristic polynomial $p(x) \in \mathbb{Q}[x]$. If $\gamma$ is a root of $p(x)$, then for any degree $d \in \nn_0$, the coefficients of the $x^d$ terms attached to conjugates of $\gamma$ in the closed form of $a_j$ are themselves conjugates.
\end{lemma}

\begin{proof}
    Take the explicit form of our sequence to be
    \[
        a_j = \sum_{i=1}^r \lambda_i(j) \gamma_i^j,
    \]
    where $\gamma_1, \gamma_2, \dots, \gamma_r \in \cc$ are the distinct roots of $p(x)$ and $\lambda_i(x) \in \mathbb{C}[x]$ is the polynomial coefficient to $\gamma_i$ having degree less than the multiplicity of $\gamma_i$ in $p(x)$. Let $S$ be the set of all roots of $p(x)$ and the various coefficients in $\lambda_i(x)$. All roots are clearly algebraic, but the same holds for the coefficients, as they can theoretically be solved for through Cramer's rule, which would only involve algebraic numbers and operations under which the set of algebraic numbers is closed.

    Let $L/\mathbb{Q}$ be a Galois field extension containing all roots and coefficients, which may be found simply as the Galois closure of $\mathbb{Q}(S)$. Consider an arbitrary $\sigma \in \mathrm{Gal}(L/\mathbb{Q})$. For a given $a_j \in \qq$, we see that
    \[
        \sigma(a_j) = \sum_{i=1}^r \sigma(\lambda_i(j)) \sigma(\gamma_i)^j.
    \]
    Clearly, $\sigma(a_j)=a_j$ by our assumption that $a_j\in\mathbb{Q}$. Subtracting $\sigma(a_j)$ from $a_j$ gives
    \[
        \sum_{i=1}^r\lambda_i(j)\gamma_i^j - \sum_{i=1}^r \sigma(\lambda_{i}(j)) \gamma_{\tau(i)}^j = 0.
    \]
    For simplicity, suppose that $\tau \colon \llbracket1,r\rrbracket \to\llbracket1,r\rrbracket$ is defined so that $i \xmapsto{\tau}j$ if $\gamma_i \xmapsto{\sigma} \gamma_j$. We can therefore combine terms by applying $\tau^{-1}$ on the indices of the latter sum, which preserves our expression as $\tau$ is a bijection, to obtain
    \[
        \sum_{i=1}^r (\lambda_i(j) - \sigma(\lambda_{\tau^{-1}(i)}(j))) \gamma_i^j = 0.
    \]
    Let $d_i(x) \in \cc[x]$ denote the difference $\lambda_i(x) - \sigma(\lambda_{\tau^{-1}(i)}(x))$, in which case
    \[
        a_j - \sigma(a_j) = \sum_{i=1}^r d_i(j)\gamma_i^j
    \]
    is a sequence of zeroes. However, any sequence has a canonical general form in terms of exponentials multiplied by polynomials, and since $d_i(x)$ as the zero polynomial would cause $(a_j - \sigma(a_j))_{j \in \nn_0}$ to be a sequence of zeroes as is the case, then by uniqueness it must be that $d_i(x)=0$ for each $i \in \llbracket 1,r \rrbracket$. That is, as polynomials, $\lambda_i(x) = \sigma(\lambda_{\tau^{-1}(i)}(x))$ for each $i$. The coefficients at each degree of $\lambda_i(x)$ and $\lambda_{\tau^{-1}(i)}(x)$ must then be conjugates as equality of the polynomials holds separately along each degree. Since $L$ is normal and contains the splitting field of $q(x)$, then $\mathrm{Gal}(L/\mathbb{Q})$ acts transitively on the set of conjugates of $\gamma$. That is, for any conjugate $\gamma'$ of $\gamma$, there exists $\sigma \in \mathrm{Gal}(L/\mathbb{Q})$ that maps $\gamma$ to $\gamma'$. This establishes that for any conjugate of $\gamma$, the coefficients at the same degree are themselves conjugates.
\end{proof}

As with the alternative proof that we presented in the antimatter case, we consider a recurrence relation that, given some element in $\mathbb{Z}[x]$, finds an equivalent one in $\mathbb{N}_0[x]$. The above characterization is necessary to establish the asymptotic behavior of the recurrence. However, one more lemma stands in our way. The recurrence is not applicable if the number of initial terms is greater than the order of our recurrence.

\begin{theorem}\label{thm:simple-antimatter-is-valuation}
    If $\alpha \in \aaa_{>0}$ is a root of a simple polynomial in $x\nn_0[x]-1$, then $M_\alpha$ is a valuation monoid.
\end{theorem}

\begin{proof}
    Assume that $\alpha \in \aaa_{>0}$ is a root of a simple polynomial $n(x) \in x\nn_0[x]-1$. Consider the polynomial $p(x) = r_{d-1}x^{d-1} + r_{d-2}x^{d-2} + \dots + r_1 x + r_0 \in \mathbb{Z}[x]$ for some $d\in\mathbb{N}$, and observe that $p(\alpha)$ is the general form for the difference between any two elements in $M_\alpha$. As the case of $p(\alpha)=0$ is clear, it suffices to show that $p(\alpha) \in M_\alpha$ whenever $p(\alpha) > 0$. First, Lemma~\ref{lemma:big-and-simple-antimatter-BAGA} already supplies that $\deg n(x) \ge d$. Thus, we can further assume that $\deg n(x)=d$ simply by padding $p(x)$ with coefficients of zero in case $d < \deg n(x)$. Hence, we may write $n(x) = c_0 x^d + c_1 x^{d-1} + \dots + c_{d-1}x - 1$ for some coefficients $c_0, \dots, c_{d-1} \in \nn_0$ with $c_0 \neq 0$ as having the same degree $d$.

    The rest of the proof consists of manipulating $p(x)$ into a polynomial in $\nn_0[x]$ while not changing its value when evaluated at $\alpha$, similar to that in our proof of Proposition~\ref{prop:lonely-Perroney-simple-antimatter}. Let $N$ be a large positive integer to be determined. We will exhibit a polynomial $g_N(x) \in \mathbb{Z}[x]$ such that $p(x) + n(x)g_N(x) := P_N(x) \in \mathbb N_0[x]$, which will show that $p(\alpha) = P_N(\alpha)$ is indeed an element of $M_\alpha$.

    Define
    \[
        F(x) = x^d - \sum_{i=0}^{d-1} c_ix^i, 
    \] and let $\beta=\alpha^{-1}$ be the positive root $F(x)$, where from Theorem~\ref{thm:antimatter-exact} it must be a Perron number. Set the coefficients of $g_N(x)$ to be determined as $g_N(x) := a_0 + a_1 x + a_2 x^2 + \cdots + a_N x^N$ through the two recurrences below. For each $j \in \llbracket0,d-1\rrbracket$, we define
    \begin{equation} \label{eq:antimatter is valuation first display}
        a_j := r_j + c_{d-1}a_{j-1} + \cdots + c_{d-j}a_0,
    \end{equation}
    while for each $j\in\llbracket d,N\rrbracket$, we set
    \begin{equation} \label{eq:antimatter is valuation second display}
        a_j := c_{d-1}a_{j-1} + c_{d-2}a_{j-2} + \cdots + c_0a_{j-d}.
    \end{equation}

    Observe that the sequence $(a_j)_{j=0}^N$ satisfies the linear homogeneous recurrence described in~\eqref{eq:antimatter is valuation first display} whose initial values $a_0, a_1, \dots, a_{d-1}$ are determined by $r_i$ and $c_i$ as in~\eqref{eq:antimatter is valuation second display}. The characteristic polynomial of this recursion is precisely $F(x)$. Now suppose that the distinct roots of $F(x)$ are $\beta, \gamma_1, \dots, \gamma_s \in \mathbb{C}$ for $s \in \mathbb{N}$, in which case it follows that
    \[
       a_j = \lambda \beta^j + \lambda_1(j) \gamma_1^j + \lambda_2(j) \gamma_2^j + \cdots + \lambda_s(j) \gamma_s^j
    \]
    for some complex polynomials $\lambda_i(x)\in\mathbb{C}[x]$ each having degree less than the multiplicity of $\gamma_i$. The reason that $\beta^j$ has only a constant for its coefficient is that $\beta$ has multiplicity $1$ by Descartes' Rule of Signs.

    Let us first consider the case of $\lambda\in\mathbb{R}\setminus\{0\}$. Note first the inequality $\lvert\gamma_i\rvert < \beta$ for each $i\in\llbracket1,s\rrbracket$ according to Proposition~\ref{prop:antimatter-Perron-number}. Hence, as $\lambda\neq0$, then $a_j\sim\lambda\beta^j$. Therefore, for sufficiently large $N$, the coefficients $a_{N-d}, a_{N-d+1}, \dots, a_N$ are either all positive or all negative depending on the sign of $\lambda$. As each $c_i$ is a nonnegative integer, that implies that the coefficients of $x^{N+1}, x^{N+2}, \dots, x^{N+d}$ in $P_N(x)$ are all of the same sign as well. In addition, for $j \in \llbracket 0,d-1 \rrbracket$, the coefficient of $x^j$ in $P_N(x)$ is
    \[
        [x^j]P_N(x) = r_j + c_{d-1}a_{j-1} + \dots + c_{d-j}a_0 - a_j = 0.
    \]
    In a similar manner, we can see that the coefficient of $x^j$ in $P_N(x)$ for $j \in \llbracket d,N \rrbracket$ is the following:
    \[
        [x^j] P_N(x) = c_{d-1} a_{j-1} + c_{d-2} a_{j-2} + \dots + c_0a_{j-d} - a_j = 0.
    \]
    Both of these equalities are by design. Hence, the support of $P_N(x)$ is contained within $\llbracket N+1,N+d\rrbracket$, meaning that either $P_N(x)$ or its negative lives in $\nn_0[x]$. Of course, if $P_N(x)$ has only non-positive terms, then $p(\alpha) = P_N(\alpha) \leq 0$, a contradiction. This leaves only the possibility of nonnegative terms. Hence, regardless of its sign, showing that $\lambda \in \mathbb{R} \setminus \{0\}$ will suffice.

    Meanwhile, it is easy to find a contradiction for $\lambda\not\in\mathbb{R}$. Again, $a_j$ approaches $\lambda\beta^j$ asymptotically by the dominance of $\beta$. If $\lambda$ is not a real number, $a_j$ would not be real for arbitrarily large $j$, though of course it is as we have an integer recurrence. Therefore, the only other case is when $\lambda = 0$. We proceed to argue that this leaves $p(\alpha)=0$, which contradicts our hypothesis. Lemma~\ref{lem:rational homogeneous linear recurrences}, along with the fact that the orbit of~$0$ under the action of any Galois group consists only of itself, guarantees that the coefficients of $\lambda_k$ for each $k$ such that $\gamma_k^{-1}$ is a root of $m_\alpha(x)$ (again having a multiplicity of $1$) are all zero. Hence, letting $n(x) = m_\alpha(x)f(x)$ for some $f(x) \in \mathbb{Z}[x]$, then $(a_j)_{j \ge 1}$ actually satisfies a recurrence governed solely by the reciprocal polynomial of $f(x)$. The quotient $f(x)$ has integer coefficients as a result of Gauss's lemma.

    Going back to our equation $p(x) + (f(x)m_\alpha(x))g_N(x) = P_N(x)$ from before, we see that $p(x)$ occupies the terms of small degree and $P_N(x)$ those of large degree (when $N$ is large). Specifically, the support lies inside the union $\llbracket 0, d-1 \rrbracket \cup \llbracket N+1, N+d \rrbracket$. However, as $g_N(x)$ has coefficients that already satisfy the recurrence given by the reciprocal polynomial of $f(x)$, the product $f(x)g_N(x)$ itself has only terms of small degree and large degree. After letting $d'$ be the degree of $m_\alpha(x)$, the support contained in $\llbracket 0, d-d'-1 \rrbracket \cup \llbracket N+1, N+d-d' \rrbracket$). In fact, we will explicitly decompose $f(x)g_N(x) := b(x) + c(x)$, where $b(x)$ consists of the bottom terms, i.e., $\supp b(x) \subseteq \llbracket0,d-d'-1\rrbracket$, and $c(x)$ consists of the terms of high degree, i.e., $\supp c(x) \subseteq \llbracket N+1, N+d-d' \rrbracket$. Consider now multiplying the product $f(x)g_N(x)$ by $m_\alpha(x)$, which becomes
    \[
        f(x)g_N(x)m_\alpha(x) = b(x)m_\alpha(x) + c(x)m_\alpha(x).
    \]
    Clearly, the $b(x)m_\alpha(x)$ terms have degrees in $\llbracket0, d-1\rrbracket$ while the $c(x)m_\alpha(x)$ terms have degrees in $\llbracket N+1, N+d \rrbracket$. For sufficiently large $N$, these two sets are disjoint, meaning that there must be exact correspondence. Specifically, $b(x)m_\alpha(x) = -p(x)$ (and $c(x)m_\alpha(x)=P_N(x)$). This shows that $p(x)$ is a multiple of $m_\alpha(x)$, whence $p(\alpha) = 0$.

    Hence, regardless of $\lambda$, we have found that $p(\alpha) \in M_\alpha$.
\end{proof}

We will show two examples illustrating the above argument, the first one not involving double roots and where $p(\alpha)\neq 0$.

\begin{example}
    Set $m_\alpha(x) = x^3+3x^2+2x-1$ be the minimal polynomial of its unique positive root $\alpha$. Since $m_\alpha(x)$ is already a simple polynomial and lies in $x\mathbb{N}_0[x]-1$, the conditions of Theorem~\ref{thm:simple-antimatter-is-valuation} apply with $n(x) = m_\alpha(x)$. Let us show that $4\alpha \mid_{M_\alpha} 1+3\alpha^2$, which corresponds to $p(x) = 1 - 4x + 3x^2$. As $p(\alpha) \approx 0.017453$, we will show that $p(\alpha) \in M_\alpha$. As $\deg n(x) = 3 > 2 = \deg p(x)$, our current $n(x)$ suffices. In particular, $d=3$ and $(c_0,c_1,c_2) = (1,3,2)$. Furthermore, $(r_2,r_1,r_0) = (3,-4,1)$. We now hope to find some $g_N(x)$ such that $p(x) + n(x)g_N(x) = P_N(x) \in \mathbb{N}_0[x]$. The coefficients of $g_N(x)$ satisfy the recurrence as above, with the initial terms chosen so that the early coefficients of $P_N(x)$ are zero. For instance, $a_0=1$ as that is the unique value for which $p(x)+a_0n(x)$ has no constant term. Indeed, $p(x) + n(x) = -2x + 6x^2 + x^3$, which is why $a_1 = -2$. The first several values of $a_j$ are listed.
    \begin{table}[H]
        \centering
        \begin{tabular}{||l r||}
            \hline
            $j$ & $a_j$ \\ [0.5ex]
            \hline\hline
            \rule{0pt}{0.34cm}$0$ & $1$ \\
            $1$ & $-2$ \\
            $2$ & $2$ \\
            $3$ & $-1$ \\
            $4$ & $2$ \\
            $5$ & $3$ \\
            $6$ & $11$ \\
            $7$ & $33$ \\
            \hline
        \end{tabular}
    \end{table}
    As can be seen, the terms initially oscillate for small $j$, but gradually become positive and grow without bound. In particular, we have the closed form $a_j=\lambda\beta^j+\lambda_1\gamma_1^j+\lambda_2\gamma_2^j$, where $\beta,\gamma_1$, and $\gamma_2$ are the roots and are approximately
    \begin{align*}
        \beta&\approx3.07959562349144,\\
        \gamma_1&\approx-0.539797811745719+0.182582254557443i,\\
        \gamma_2&\approx-0.539797811745719-0.182582254557443i.
    \end{align*} Meanwhile, the coefficients are approximately \begin{align*}
        \lambda&\approx0.0126035453089533,\\
        \lambda_1&\approx0.493698227345523+4.12367396014509i,\\
        \lambda_2&\approx0.493698227345523-4.12367396014509i,
    \end{align*}
    which satisfy the irreducible polynomial $23m_\lambda(x) = 23x^3 - 23x^2 + 397x - 5$. As they share a minimal polynomial, then $\lambda,\lambda_1,\lambda_2$ are conjugates, providing evidence for Lemma~\ref{lem:rational homogeneous linear recurrences}. However, this is not necessary in finding an $N$ and $P_N(x) \in \mathbb{N}_0[x]$. We see that if $N=4$ and $g(x) = 1 - 2x + 2x^2 - x^3 + 2x^4$, then
    \[
        p(x) + n(x)g(x) = 3x^5 + 5x^6 + 2x^7 = P_N(x).
    \]
    Because all the coefficients are nonnegative, one obtains that $P_N(\alpha) \in M_\alpha$. Therefore, $p(\alpha) = P_n(\alpha)$ is an element of $M_\alpha$, and $4\alpha + (3\alpha^5 + 5\alpha^6 + 2\alpha^7) = 1 + 3\alpha^2$. Hence $4\alpha \mid_{M_\alpha} 1 + 3\alpha^2$. \hfill $\blacksquare$
\end{example}

Our second example illustrates the case where $\lambda=0$, for which we prove $p(\alpha)=0$. It will also demonstrate our argument about doubling the degree.

\begin{example}
    Consider $\alpha\in\aaa$ satisfying the minimal polynomial $m_\alpha(x) = x^3-x^2+2x-1$. Although $m_\alpha(x)$ is not yet in the desired form of Theorem~\ref{thm:simple-antimatter-is-valuation}, one can readily see that
    \[
        (x+1)m_\alpha(x) = x^4+x^2+x-1 \in x\mathbb{N}_0[x].
    \]
    Let us show that $1+\alpha\mid_{M_\alpha}4\alpha^2+2\alpha^4+\alpha^5$, which corresponds to $p(x) = -1 - x + 4x^2 + 2x^4 + x^5$. In this case, $p(x) = m_\alpha(x)(x^2+3x+1)$, so $p(\alpha)=0$ and the two sides are the same. We will illustrate how this plays out in the above proof. The polynomial we found earlier, $x^4+x^2+x-1$, does not yet have a degree greater than that of $p(x)$, so we will need to double its degree at least once. Notice that
    \[
        (x^4+x^2+x-1)(x^4+1) = x^8+x^6+x^5+x^2+x-1,
    \]
    which clearly remains simple. As the degree of this new polynomial is greater than that of $p(x)$, we may take $n(x) = x^8 + x^6 + x^5 + x^2 + x - 1$ and $d=8$. Performing the recurrence produces some selected terms as follows.
    \begin{table}[H]
        \centering
        \begin{tabular}{||l r||}
            \hline
            $j$ & $a_j$ \\ [0.5ex]
            \hline\hline
            \rule{0pt}{0.34cm}$0$ & $-1$ \\
            $1$ & $-2$ \\
            $2$ & $1$ \\
            $3$ & $-1$ \\
            $4$ & $2$ \\
            $5$ & $1$ \\
            $6$ & $0$ \\
            $7$ & $0$ \\
            $8$ & $-1$ \\
            $9$ & $-2$ \\
            \hline
        \end{tabular}
        \hspace{1em}
        \begin{tabular}{||l r||}
            \hline
            $j$ & $a_j$ \\ [0.5ex]
            \hline\hline
            \rule{0pt}{0.969em}$88$ & $-1$ \\
            $89$ & $-2$ \\
            $90$ & $1$ \\
            $91$ & $-1$ \\
            $92$ & $2$ \\
            $93$ & $1$ \\
            $94$ & $0$ \\
            $95$ & $0$ \\
            $96$ & $-1$ \\
            $97$ & $-2$ \\
            \hline
        \end{tabular}
    \end{table}
    \noindent The low magnitude of terms after nearly $100$ of them indicates that $\lambda=0$, and the apparent periodicity indicates they may never become completely positive. Further, when $N=100$,
    \[
         P_N(x) = x^{101} - x^{102} - x^{103} - x^{104} - x^{105} + 3x^{106} - x^{107} + 2x^{108},
    \]
    and the many negative terms suggests that the recurrence is unlikely ever to yield a polynomial in $\mathbb{N}_0[x]$. Of course, even though this particular sequence is periodic, this is not in general true. Hence the easiest solution would be for $p(\alpha) = 0$, and that is precisely what we will proceed to demonstrate systematically. First, we compute the coefficients, and, indeed, we find that the general form is
    \begin{align*}
        a_j&=\left(\left(-\frac{1}{8}+\frac{\sqrt{2}}{4}\right)i-\left(\frac{3}{8}+\frac{\sqrt{2}}{8}\right)\right)\left(\frac{\sqrt{2}}{2}+\frac{\sqrt{2}}{2}i\right)^j \\
        &+\left(\left(\frac{1}{8}-\frac{\sqrt{2}}{4}\right)i-\left(\frac{3}{8}+\frac{\sqrt{2}}{8}\right)\right)\left(\frac{\sqrt{2}}{2}-\frac{\sqrt{2}}{2}i\right)^j \\
        &+\left(\left(\frac{1}{8}+\frac{\sqrt{2}}{4}\right)i-\left(\frac{3}{8}-\frac{\sqrt{2}}{8}\right)\right)\left(-\frac{\sqrt{2}}{2}+\frac{\sqrt{2}}{2}i\right)^j \\
        &+\left(\left(-\frac{1}{8}-\frac{\sqrt{2}}{4}\right)i-\left(\frac{3}{8}-\frac{\sqrt{2}}{8}\right)\right)\left(-\frac{\sqrt{2}}{2}-\frac{\sqrt{2}}{2}i\right)^j \\
        &+\left(\dfrac{1}{2}\right)(-1)^j,
    \end{align*}
    and as we have roots of unity, there will indeed be periodicity (in particular, as the least common multiple of the orders is $8$, then the period is $8$ as seen in the table). We will then note that none of the roots with nonzero coefficients are roots to $r_\alpha(x)$. In fact, as soon as we observe that $\lambda$, the coefficient of $\beta$, is equal to zero, we are guaranteed that all other coefficients of conjugates of $\beta$ are zero by Lemma~\ref{lem:rational homogeneous linear recurrences}. This implies that $g_N(x)$ actually satisfies a restricted recurrence of
    \[
        \frac{n(x)}{m_\alpha(x)} = (x+1)(x^4+1) = x^5+x^4+x+1.
    \]
    We denote the quotient by $f(x)$, and observe that $f(x)$ is an integer polynomial, which is guaranteed by Gauss's lemma. Indeed, for some large $N$ such as $N=100$, we obtain that
    \[
        g_N(x) f(x) = 2x^{105} + x^{104} - x^{102} - x^{101} - x^2 - 3x - 1.
    \]
    As expected, we see both high terms and low terms, and nothing in between. Here we would decompose $b(x)=-x^2-3x-1$ and $c(x)=2x^{105}+x^{104}-x^{102}-x^{101}$, where $\supp b(x)\in\llbracket0,8-3-1\rrbracket$ and $\supp c(x)\in\llbracket100+1,100+8-3\rrbracket$. Indeed, $-b(x)=x^2+3x+1$ is recognizable from above, and we do see that $-b(x)m_\alpha(x) = p(x)$. This follows systematically from the fact that $p(x)$ occupies the small terms and $P_N(x)$ the large ones. Therefore, $p(\alpha)=0$, as desired. \hfill $\blacksquare$
\end{example}

The two examples illustrate the two main cases of the proof being quite intricate and essential. In particular, it naturally yields the following results.

\begin{theorem}
    For any $\alpha \in \aaa \cap (0,1)$, the following conditions are equivalent.
    \begin{enumerate}
        \item [(1a)] $M_\alpha$ is a simple antimatter monoid.
        \smallskip

        \item [(1b)] $\alpha^{-1}$ is a Perron number with no positive conjugate.
        \smallskip

        \item [(1c)] There exists a simple polynomial $n(x) \in x\nn_0[x]-1$ satisfying $n(\alpha) = 0$.
        \smallskip

        \item [(1d)] $M_\alpha$ is a valuation monoid.
        \smallskip

        \item [(1e)] $M_\alpha$ is a simple GCD monoid.
    \end{enumerate}
    \smallskip

    \noindent For any algebraic $\alpha \in (0,1)$, the following conditions are also equivalent. 
    \begin{enumerate}
        \item [(2a)] $M_\alpha$ is an antimatter monoid.
        \smallskip

        \item [(2b)] The simplified polynomial of $r_\alpha(x)$ has a root that is a Perron number and $\alpha$ has no positive conjugate aside from itself. 
        \smallskip

        \item [(2c)] There exists a polynomial $n(x) \in x\nn_0[x]-1$ satisfying $n(\alpha) = 0$.
        \smallskip

        \item [(2d)] $M_\alpha$ is the product of valuation monoids. In particular, if $n=\gcd \supp m_\alpha(x)$, then $M_{\alpha^n}$ is a valuation monoid and $M_\alpha\cong M_{\alpha^n}^n$.
        \smallskip

        \item [(2e)] $M_\alpha$ is a GCD monoid.
    \end{enumerate}
\end{theorem}

The first set of equivalences mimics Theorem~\ref{thm:antimatter-exact}, while the second Corollary~\ref{cor:antimatter-exact}.

\begin{proof}
    The equivalence of (1a), (1b), and (1c) is by Theorem~\ref{thm:antimatter-exact}, while that of (2a), (2b), and (2c) follows from Corollary~\ref{cor:antimatter-exact}. Meanwhile, (1c) implies (1d) by Theorem~\ref{thm:simple-antimatter-is-valuation}, which itself implies (1a) from the fact that nontrivial products (products where neither element is a group) are never valuation, which forces $M_\alpha$ to be simple, as well as $\alpha\in(0,1)$, which ensures $M_\alpha$ is antimatter. A similar set of results shows that (2a), (2b), (2c), and (2d) are equivalent.

    Finally, to show our results about GCD monoids, observe that an atomic GCD monoid is necessarily a UFM. However, the factorial $M_\alpha$ were characterized in~\cite[Theorem 5.4]{CG22}; in particular, for algebraic $\alpha$, $M_\alpha$ cannot be factorial when $0<\alpha<1$. As a result, the given restrictions on $\alpha$ demonstrate that (2e) implies (2a), and likewise for (1e) and (1a). That being said, it is routine to show that every valuation monoid is a GCD monoid, which shows (1d) implies (1e). Likewise, (2d) implies (2e) because the product of GCD monoids is a GCD monoid.
\end{proof}


\medskip
\subsection{The Valuation Set}

Given the exact characterization of both the antimatter and valuation $M_\alpha$, it is natural to turn our attention to the class of antimatter or valuation monoids as a whole. Our first result follows from the trivial fact that when $\alpha$ has no positive conjugate, $M_\alpha$ is an abelian group and, hence, both antimatter and valuation. We remark that, as each transcendental number generates an atomic monoid with infinitely many atoms, the set of antimatter or valuation monoids generated in this way is at most countable.

\begin{proposition}
    The set of $\alpha \in \mathbb{C}$ such that the monoid $M_\alpha$ is a valuation (alternatively, antimatter) is dense in the complex plane. 
\end{proposition}

\begin{proof}
    Let $S$ be the set consisting of all Gaussian rationals that are not nonnegative real numbers, $S = (\mathbb{Q}+i\mathbb{Q})\setminus\mathbb{R}_{\ge 0}$. As the rank of each $\alpha \in S$ is either one or two, the algebraic conjugates of $\alpha$ are either negative (when $\alpha \in \qq_{<0}$) or nonreal (when $\alpha \in \cc \setminus \rr$). In either case, $\alpha$ has no positive conjugates. Therefore, $M_\alpha$ is an abelian group for all $\alpha \in S$ in light of~\cite{aD07}. As $S$ is dense in $\cc$, the fact that every abelian group is, in a trivial way, both a valuation monoid and an antimatter monoid concludes our proof.
\end{proof}

We may generalize to density in terms of minimal polynomials as opposed to in the complex plane. In particular, the literature contains several results about the distribution of polynomials with no real roots. This gives a lower bound for the measure of polynomials with no positive root, which itself yields the trivial examples of valuation and antimatter monoids. For $d, s \in \nn$, we let $V_d^\ast(s)$ denote the set consisting of all vectors $(c_0, c_1, \dots, c_d) \in \mathbb{Z}^{d+1}$ such that $c_d x^d + \dots + c_1 x + c_0 \in \zz[x]$ is a degree-$d$ polynomial with exactly $2s$ nonreal roots (counting multiplicity). To later compute our limiting density, we begin with a parameter $B \in \rr_{\ge 1}$, which represents either a bound on the coefficients or on the roots---making the problem more tractable. First, let $\mathcal{D}_d^\ast(s, B)$ denote the subset of $V_d^\ast(s)$ with height bounded by $B$, i.e., $\mathcal{D}_d^\ast(s,B) = V_d^\ast(s) \cap \ldb -B, B \rdb^{d+1}$ where each coefficient has absolute value at most $B$. Meanwhile, $\mathcal{N}_d^\ast(s,B)$ denotes those polynomials each of whose roots is of distance at most $B$ from the origin. If we further restrict to monic polynomials, we drop the asterisk, giving $\mathcal{D}_d(s,B)$ and $\mathcal{N}_d(s,B)$.

A natural way to define the density, then, is to take the ratio $$D_d^\ast=\lim_{B\to\infty}\frac{|\mathcal{D}_d^\ast(\lfloor d/2\rfloor,B)|}{\lvert\mathcal{D}_d^\ast(0,B)\rvert+\lvert\mathcal{D}_d^\ast(1,B)\rvert+\cdots+\lvert\mathcal{D}_d^\ast(\lfloor d/2\rfloor,B)\rvert}$$ or the alternatives $D_d$, $N_d^\ast$, $N_d$ defined without the star or by replacing $\mathcal D$ with $\mathcal N$, each of which counts the proportion of polynomials satisfying the given polynomial restraints that also satisfy the constraint on the number of roots, namely having at most one real root. In particular, when $d$ is even, then $\lfloor d/2\rfloor=d/2$, so each polynomial represented in the numerator will have no real roots. When $d$ is odd, there will be one real root, but negating the coefficients pairs each polynomial with a positive root with one with a negative root. 

\begin{proposition}
    The following statements hold. \begin{enumerate}
        \item [(1)] $D_d=0$ when $d$ is even.
        \smallskip
        
        \item [(2)] $D_d^\ast>0$ for all $d$. Further, when $d$ is odd, $D_{d+1}^\ast=D_d$.
        \smallskip
        
        \item [(3)] $N_d>0$ for all $d$. Further, when $d$ is even, $N_d$ is asymptotic to $c_dd^{-3/8}$ for some $c_d>0$.
    \end{enumerate}
\end{proposition}

\begin{proof}
    (1) follows from~\cite[Theorem~1.1]{aD18}, while (2) follows from~\cite[Theorem~2.1]{BHP17} and~\cite[Theorem 1.2]{aD18}. Meanwhile, (3) follows from combining~\cite[Corollary~3.2]{AP14b} with~\cite[Theorem~6.1]{AP14a}.
\end{proof}

Therefore, the limiting density of the class of antimatter and valuation monoids is nonzero for a variety of definitions of limiting density. 
As these results readily follow from what is already known, we restrict our discussion to nontrivial $M_\alpha$ by taking $\alpha \in (0,1)$. Then we set
\[
	V := \big\{ \, \alpha \in (0,1) :  M_\alpha  \text{ is a valuation monoid } \big\},
\]
and we prove a variety of results about the size and density of $V$ even in this nontrivial case.

\begin{theorem}
    For each $d \in \nn$, there exist infinitely many pairwise non-isomorphic valuation monoids $M_\alpha$ having rank~$d$. 
\end{theorem}

\begin{proof}
    For $d=1$, simply taking $\alpha^{-1}\in\mathbb{N}$ gives an infinite number of valuation monoids. For $d>1$, let $\alpha^{-1}$ be the the unique positive root to $x^d-ax-1$, where $a\in\mathbb{N}\setminus\{2\}$. By~\cite[Theorems 1 and 2]{eS56}, $x^d-ax-1$ is irreducible in each of those cases, making it the minimal polynomial of $\alpha^{-1}$. We may easily verify from Theorem~\ref{thm:simple-antimatter-is-valuation} that each case yields a valuation monoid; in particular, Descartes' rule of signs guarantees that $\alpha^{-1}$ has no positive conjugate. Further, $\alpha^{-1}$ is an algebraic integer and at least each of its conjugates by norm through an argument analogous to the one presented in the proof of Proposition~\ref{prop:antimatter-Perron-number}. In fact, $m_{\alpha^{-1}}(x)$ being simple means $\alpha^{-1}$ is a Perron number by~\cite{dB94}. Moreover, varying $a$ again yields infinitely many valuation monoids, which are distinct because their corresponding minimal polynomials are never unit multiples of one another.
\end{proof}

Next we prove that $V$ is dense in the real interval $(0,1)$.

\begin{theorem}
    The set $V$ is dense in $(0,1)$.
\end{theorem}
	
\begin{proof}
	For each pair $(d,n) \in \nn \times \nn_{\ge 2}$, we set $P_{d,n}(x) := x^d - \frac1n \in \zz[x]$ and note that $P_{d,n}(x)$ has only one positive root, namely, $\sqrt[d]{1/n}$. It turns out that the set consisting of all such roots is dense in $(0,1)$.
	\smallskip
	
 	\noindent \textsc{Claim.} The set $\Big\{\ \sqrt[d]{\frac{1}{n}} : n,d \in \mathbb{N} \Big\}$ is dense in $(0,1)$.
	\smallskip

	\noindent \textsc{Proof of Claim.} Intuitively, as $d$ grows larger, the maximum difference between the $d\textsuperscript{th}$ roots of consecutive unit fractions tends to $0$. In particular, suppose $a\in(0,1)$ and $\varepsilon>0$. Take $d$ large enough so that $$1-\sqrt[d]{\frac{1}{2}}<\varepsilon.$$ The distance between the $d\textsuperscript{th}$ roots of $\frac1n$ and $\frac1{n+1}$ is maximized when $n=1$, so the distance between the $d\textsuperscript{th}$ roots of any two consecutive unit fractions is less than $\varepsilon$. Thus, the minimum value of $\lvert a-\sqrt[d]{1/n}\rvert$ across positive integers $n$ is less than $\varepsilon$ because the roots range from arbitrarily close to $1$ when $d=1$ and arbitrarily close to $0$ when $d$ is large.  \hfill $\square$
	\smallskip
	
	While the choice of parameters $\sqrt[d]{1/n}$ produces antimatter but not necessarily valuation monoids (see Example~\ref{ex:the monoids M_{d,p}}), a slight modification in our choice of parameters will. Consider, for each pair $(d,n) \in \nn \times \nn_{\ge 2}$, the polynomial
	\[
		Q_{d,n}(x) := (n-1)x^d+x^{d-1} - 1 \in \zz[x],
	\]
	and notice that $Q_{d,n}$ has a unique positive root, which exhibits a close similarity to the minimal polynomial $x^d-\frac1n$ of $\sqrt[d]{1/n}$. For each triple $(k,d,n) \in \nn^2 \times \nn_{\ge 2}$, we let $\alpha_{k,d,n}$ denote the unique positive root of $Q_{kd,n^k}(x)$ and we will check that $\lim_{k \to \infty} \alpha_{k,d,n} = \sqrt[d]{1/n}$. We proceed to argue that
    \[
    	\lim_{k \to \infty} \alpha_{k,d,n} = \sqrt[kd]{\frac1{n^k}} = \sqrt[d]{\frac1{n}}.
    \]
    It suffices simply to bound the difference of their inverses. First, notice that after evaluating the reciprocal polynomial $x^{kd}-x-(n^k-1)$ at $\sqrt[d]{n}$ yields the negative value $1-\sqrt[d]{n}$. Hence $\sqrt[d]{n}<\alpha_{k,d,n}^{-1}$ by the Intermediate Value Theorem. However, evaluating instead at $\sqrt[d]{n}\,\big(1+\frac{1}{kd}\big)$ yields, by a truncation of the binomial expansion after the first two terms, a value greater than
    \[
    	n^k - \bigg( \sqrt[d]{n} + \frac{\sqrt[d]{n}}{kd} \bigg) + 1.
    \]
    This value is positive for sufficiently large $k$, meaning that $\alpha_{k,d,n}^{-1} < \sqrt[d]{n} \, \big(1 + \frac1{kd} \big)$, again following from the Intermediate Value Theorem. For large $k$, the two bounds are arbitrarily close together. Thus, the inverse of $\alpha_{k,d,n}$ approaches the inverse of our target. At no point is either the limit or $\alpha_{k,d,n}$ equal to zero, meaning we can reciprocate and extract that $\alpha_{k,d,n}$ approaches $\sqrt[d]{1/n}$ as $k \to \infty$.

    We proceed to argue that, for each triple $(k,d,n) \in \nn^2 \times \nn_{\ge 2}$, the monoid $M_{\alpha_{k,d,n}}$ has the valuation property. To do this, first observe that $\alpha_{k,d,n}^{-1}$ is a root of the polynomial $x^{kd}-x-(n^k-1)$, which makes it an algebraic integer. As its minimal polynomial has one sign change, $\alpha_{k,d,n}^{-1}$ has no positive conjugates aside from itself. Moreover, being simple with only its leading coefficient positive, $\alpha_{k,d,n}^{-1}$ is in fact a Perron number by an analogue to Proposition~\ref{prop:antimatter-Perron-number}. Thus, $M_{\alpha_{k,d,n}}$ is a valuation monoid. Hence
	\[
		A := \{\alpha_{k,d,n} : (k,d,n) \in \nn^2 \times \nn \} \subseteq V.
	\]
	On the other hand, the fact that $\lim_{k \to \infty} \alpha_{k,d,n} = \sqrt[d]{1/n}$ for all fixed pair~$(d,n) \in \nn \times \nn_{\ge 2})$ guarantees that the set $\big\{\sqrt[d]{1/n} : (d,n) \in \nn \times \nn_{\ge 2} \}$ is contained in the closure of~$A$. Thus, by virtue of our established claim, $V$ must be dense in the interval $(0,1)$.
\end{proof}

We continue with some results about the structural properties of $V$. Observe that $V^{-1}$, the set of $\alpha^{-1}$ for $\alpha\in V$, is contained within the Perron numbers (Proposition~\ref{prop:antimatter-Perron-number}), which is closed under addition and multiplication~\cite[Proposition 1]{dL84}. However, $V^{-1}$ is closed under neither as shown by the following two examples, and this results from the fact that the property of having no distinct positive conjugates is very rarely preserved under either operation. In particular, in our discussion below, we only need to check whether the sum or product has positive algebraic conjugates distinct from itself.

\begin{example}
    One may easily verify that $\alpha=\sqrt{2}-1$, with minimal polynomial $m_\alpha(x)=x^2+2x-1$, is in $V$. However, $\alpha^2=3-2\sqrt{2}\not\in V$ as $3+2\sqrt{2}$ is a distinct positive conjugate. \hfill $\blacksquare$
\end{example}

Meanwhile, a counterexample that $V^{-1}$ is not closed under addition is more involved. In fact, although we were able to square $\alpha$ earlier, dividing $\alpha$ by two (the equivalent in the additive case) will not affect the valuation property.

\begin{example}
    Let $\beta$ be defined as the reciprocal of $\alpha$ in the above example (with minimal polynomial $m_\beta(x)=x^2-2x-1$) and consider the golden ratio $\varphi=(1+\sqrt{5})/2$ satisfying $m_\varphi(x)=x^2-x-1$. To show that $V^{-1}$ is not closed, it suffices to demonstrate that $\beta+\varphi$ has a positive conjugate distinct from itself. In particular, we find via the resultant that the minimal polynomial of $\beta+\varphi$ is $x^4-6x^3+7x^2+6x-9$. While $\beta+\varphi$ remains a Perron number, this polynomial has two additional positive roots, which means that $\beta+\varphi$ has a positive algebraic conjugate aside from itself. This precludes $M_{(\beta+\varphi)^{-1}}$ from being valuation. \hfill $\blacksquare$
\end{example}

Hence we proceed by finding several sufficient and several necessary conditions about when it is possible to multiply two elements in $V$ or add two inverses in $V^{-1}$.

\begin{theorem}
    For $\alpha, \beta \in \aaa$, let $M_\alpha$ and $M_\beta$ be valuation monoids. Then the following statements hold.
    \begin{enumerate}
        \item [(1)] If $\mathbb{Q}(\alpha)$ and $\mathbb{Q}(\beta)$ are linearly disjoint over $\mathbb{Q}$ (meaning whenever a finite subset $S\subset\mathbb{Q}(\alpha)$ is linearly independent over $\mathbb{Q}$, then it is also linearly independent over $\mathbb{Q}(\beta)$), then $M_{\alpha\beta}$ has the valuation property only if at most one of $\alpha$ or $\beta$ has a negative conjugate and at most one has a purely imaginary conjugate. Further, this becomes exact if, whenever $\gamma$ and $\delta$ are nonreal non-imaginary algebraic conjugates of $\alpha$ and $\beta$, respectively, then $\gamma\delta\not\in\mathbb{R}^+$.
        \smallskip

        \item [(2)] If in fact the splitting fields (the smallest field containing each root to the respective polynomial) of $m_\alpha(x)$ and $m_\beta(x)$ are linearly disjoint, then $M_{\alpha\beta}$ has the valuation property if and only if at most one of $\alpha$ or $\beta$ has a negative conjugate and at most one has a purely imaginary conjugate.
    \end{enumerate}
\end{theorem}

\begin{proof}
    (1) Linear disjointness ensures that every product $\gamma\delta$ is an algebraic conjugate of $\alpha\beta$, where $\gamma$ is conjugate to $\alpha$ and $\delta$ to $\beta$. Hence, if $\alpha$ and $\beta$ both have negative conjugates, then the product of these negative conjugates yields a positive conjugate of $\alpha\beta$. Similarly, if they both have purely imaginary conjugates, then by complex conjugation both $m_\alpha(x)$ and $m_\beta(x)$ have at least two imaginary roots; multiplying two on opposite sides of the real line will lead to a positive root. The other direction follows similarly.
    \smallskip

    (2) This condition is stronger than (1), so to prove the equivalence it suffices to show that if $\gamma$ is a nonreal conjugate of $\alpha$ and $\delta$ likewise of $\beta$, then $\gamma\delta\not\in\mathbb{R}^+$ as the real case has been dealt with. In particular, if $\gamma\delta\in\mathbb{R}^+$, then $\gamma\delta\in\mathbb{R}$ implies $\gamma\delta=\bar{\gamma}\bar{\delta}$, or $\gamma/\bar{\gamma} = \bar{\delta}/\delta$. However, $\gamma/\bar{\gamma}\in L_\alpha$ while $\bar{\delta}/\delta\in L_\beta$, where $L_\alpha$ and $L_\beta$ denote the splitting fields of $m_\alpha(x)$ and $m_\beta(x)$, respectively. By linear disjointness, the only common subfield of $L_\alpha$ and $L_\beta$ is $\mathbb{Q}$ itself, which means $\gamma/\bar{\gamma}$ (as well as $\bar{\delta}/\delta$) is rational. This is clearly a contradiction if $\gamma$ is non-real. In particular, if $\gamma$ is not purely imaginary, the ratio is not real, precluding rationality.
\end{proof}

\begin{example}
    As a quick example, observe that when $\alpha=\frac{1}{n}$, the splitting field of $m_\alpha(x)$ is $\mathbb{Q}\big(\frac1n\big)=\mathbb{Q}$ for any $n\in\mathbb{N}$. Hence, $\mathbb{Q}\big(\frac1n\big)$ and $K$ are trivially linearly disjoint over $\mathbb{Q}$ for any field $K$, making it quick to show that whenever $\beta\in V$, so is $\beta/n$. \hfill $\blacksquare$
\end{example}

\begin{remark}
    Fix $\alpha, \beta \in \aaa$, and observe that linear disjointness of $\mathbb{Q}(\alpha)$ and $\mathbb{Q}(\beta)$ is equivalent to the following equality:
    \[
    	[\mathbb{Q}(\alpha)\mathbb{Q}(\beta):\mathbb{Q}]=[\mathbb{Q}(\alpha):\mathbb{Q}]\cdot[\mathbb{Q}(\beta):\mathbb{Q}],
    \]
    where $\mathbb{Q}(\alpha)\mathbb{Q}(\beta)$ is the compositum field. As a consequence, if $\deg m_\alpha(x)$ and $\deg m_\beta(x)$ are coprime, then linear disjointness follows from the fact that $\mathbb{Q}(\alpha)$ and $\mathbb{Q}(\beta)$ are both subfields of $\mathbb{Q}(\alpha)\mathbb{Q}(\beta)$, which implies that
    \[
    	[\mathbb{Q}(\alpha):\mathbb{Q}] = \deg m_\alpha(x) \quad \text{ and } \quad [\mathbb{Q}(\beta):\mathbb{Q}]=\deg m_\beta(x),
    \]
    which equal the degrees of the minimal polynomials, are both factors of $[\mathbb{Q}(\alpha)\mathbb{Q}(\beta):\mathbb{Q}]$. Hence this gives a rather simple test for showing that a given product $\alpha\beta$ does not generate a valuation monoid---simply counting negative and imaginary roots may sometimes be enough to preclude $M_{\alpha\beta}$ from being valuation.
\end{remark}

On the other hand, this idea of coprime degrees is interesting as, aside from $f(x)=x$, an irreducible polynomial with odd degree cannot have a purely imaginary nonzero root. In particular, for $\xi\in i\mathbb{R}$, both $m_\xi(x)$ and $-m_\xi(-x)$ are monic irreducible polynomials of the same degree. Moreover, $\xi$ is a root to each as $\bar{\xi}=-\xi$, so the two polynomials must be the same. Hence, $m_\xi(x)$ is an odd polynomial, meaning it must be a multiple of $x$. Being irreducible, $m_\xi(x)=x$, which implies $x=0$.

In the case where the degrees of $m_\alpha(x)$ and $m_\beta(x)$ are not coprime, the problem becomes much more difficult. In particular, interactions between $\alpha$ and $\beta$ may mean that not all products of conjugates of $\alpha$ and $\beta$ become conjugates of $\alpha\beta$, which makes it difficult to make universal statements. 

Let us now move to addition. Specifically, given valuation $M_\alpha$ and $M_\beta$, we are interested in the circumstances under which $M_{(\alpha^{-1}+\beta^{-1})^{-1}}$ has the valuation property. For instance, this holds when $\alpha$ and $\beta$ are unit fractions, but we can establish a more general proposition with $\alpha\in V$.

\begin{proposition}
    For $\alpha \in V$ and $\beta^{-1}$ a unit fraction, $(\alpha^{-1} + \beta^{-1})^{-1} \in V$ if and only if $\beta^{-1} < \inf \gamma$, where $\gamma$ ranges over all negative conjugates of $\alpha^{-1}$ and the infimum of the empty set is defined to be $\infty$.
\end{proposition}

\begin{proof}
    The Perron numbers are closed under addition, so it suffices to show that $\alpha^{-1}+\beta^{-1}$ has no positive conjugates. As $\beta$ is rational, the conjugates of $\alpha^{-1}+\beta^{-1}$ are simply $\beta^{-1}$ added to a conjugate of $\alpha^{-1}$, so the our inequality ensures that no conjugates become positive.
\end{proof}

We prove one last result showing that the subset of $V$ with a given bounded degree is discrete inside the interval $(0,1)$.

\begin{proposition}
    The subset of $V$ with a given bounded degree is discrete in $(0,1)$.
\end{proposition}

\begin{proof}
    This follows from \cite[Proposition 3]{dL84}, which tells us that the Perron numbers with degree at most some fixed value are discrete in $[1,\infty)$. Further, $0\not\in V$ by definition, so even though $V$ accumulates at $0$, it remains discrete.
\end{proof}

It would be interesting to study the distribution of $V$ under a given bound on the degree. While of course they do cluster near $0$, we might ask how quickly the proportion falls off away from $0$.

\bigskip
\section*{Acknowledgments}

During the preparation of this paper, the authors were part of PRIMES-USA at MIT, and they
would like to thank the program for making this collaboration possible. Finally, the second author kindly acknowledges partial support from the NSF under the award DMS-2213323.

\bigskip


\begin{thebibliography}{20}
    \bibitem{aA56} A. C. Aitken, \emph{Determinants and Matrices}, 9th ed., Interscience Pub. (1956).

    \bibitem{ABLST23} K. Ajran, J. Bringas, B. Li, E. Singer, and M. Tirador, \emph{Factorization in additive monoids of evaluation polynomial semirings}, Comm. Algebra \textbf{51} (2023) 4347--4362.

    \bibitem{AP14a} S. Akiyama and A. Peth\H{o}, \emph{On the distribution of polynomials with bounded roots I. Polynomials with real coefficients}, J. Math. Soc. Japan \textbf{66} (2014) 927--949. 

    \bibitem{AP14b} S. Akiyama and A. Peth\H{o}, \emph{On the distribution of polynomials with bounded roots II. Polynomials with integer coefficients}, Unif. Distrib. Theory \textbf{9} (2014) 5--19. 

    \bibitem{ABP21} S. Albizu-Campos, J. Bringas, and H. Polo, \emph{On the atomic structure of exponential Puiseux monoids and semirings}, Comm. Algebra \textbf{49} (2021) 850--863.

    \bibitem{AAZ90} D.~D. Anderson, D.~F. Anderson, and M. Zafrullah, \emph{Factorizations in integral domains}, J. Pure Appl. Algebra \textbf{69} (1990) 1--19.

    \bibitem{BHP17} C. Bert\'ok, L. Hajdu, and A. Peth\H{o}, \emph{On the distribution of polynomials with bounded height}, J. Number Theory \textbf{179} (2017) 172--184.

    \bibitem{dB94} D. W. Boyd, \emph{Irreducible polynomials with many roots of maximal modulus}, Acta Arith. \textbf{68} (1994) 85--88.

    \bibitem{aB50} A. Brauer, \emph{On algebraic equations with all but one root in the interior of the unit circle. To my teacher and former colleague Erhard Schmidt on his 75th birthday.}, Math. Nachr. \textbf{4} (1950) 250--257.

    \bibitem{CGG20} S.~T. Chapman, F. Gotti, and M. Gotti, \emph{Factorization invariants of Puiseux monoids generated by geometric sequences}, Comm. Algebra \textbf{48} (2020) 380--396.

    \bibitem{CGGP25} S.~T. Chapman, F. Gotti, M. Gotti, and H. Polo \emph{On three families of dense Puiseux monoids}. In: Ideal Theory and Arithmetic of Rings, Monoids, and Semigroups (Proceedings of the UMI-AMS Special Session at Palermo). Preprint on arXiv: https://arxiv.org/abs/1701.00058

    \bibitem{pC68} P. M. Cohn, \emph{Bezout rings and their subrings}, Proc. Cambridge Philos. Soc. \textbf{64} (1968) 251--264.

    \bibitem{CG22} J. Correa-Morris and F. Gotti, \emph{On the additive structure of algebraic valuations of polynomial semirings}, J. Pure Appl. Algebra \textbf{226} (2022) 107104.

    \bibitem{CDM99} J. Coykendall, D. E. Dobbs, and B. Mullins, \emph{On integral domains with no atoms}, Comm. Algebra \textbf{27} (1999) 5813--5831.

    \bibitem{CFR05} P. Cull, M. Flahive, and R. Robson, \emph{Difference Equations: From Rabbits to Chaos}, Undergrad. Texts Math. \textbf{111} (2005).

    \bibitem{aD07} A. Dubickas, \emph{On roots of polynomials with positive coefficients}, Manuscripta Math. \textbf{123} (2007) 353--356.

    \bibitem{aD18} A. Dubickas, \emph{On the number of monic integer polynomials with given signature}, Arch. Math. \textbf{110} (2018) 333--342.


    \bibitem{sE04} S. N. Elaydi, \emph{An Introduction to Difference Equations} (3rd ed.), Undergrad. Texts Math. (2004).



    \bibitem{GGT21} A. Geroldinger, F. Gotti, and S. Tringali, \emph{On strongly primary monoids, with a focus on Puiseux monoids}, J. Algebra \textbf{567} (2021) 310--345.




    \bibitem{GG18} F. Gotti and M. Gotti, \emph{Atomicity and boundedness of monotone Puiseux monoids}, Semigroup Forum \textbf{96} (2018) 536--552.

    \bibitem{GL22a} F. Gotti and B. Li, \emph{Divisibility and a weak ascending chain condition on principal ideals}. Preprint on arXiv: \href{https://arxiv.org/abs/2212.06213}{\textcolor{black}{https://arxiv.org/abs/2212.06213}}.


    \bibitem{dH92} D. Handelman, \emph{Spectral radii of primitive integral companion matrices and log concave polynomials}, Symbolic Dynamics and its Applications, Contemp. Math. \textbf{135} (1992) 223--228.

    \bibitem{JLZ23} N. Jiang, B. Li, and S. Zhu, \emph{On the primality and elasticity of algebraic valuations of cyclic free semirings}, Internat. J. Algebra and Comput. \textbf{33} (2023) 197--210.


    \bibitem{dL84} D. A. Lind, \emph{The entropies of topological Markov shifts and a related class of algebraic integers}, Ergod. Th. \& Dynam. Sys. \textbf{4} (1984), 283--300.

    \bibitem{PW23} H. Parks and D. Wills, \emph{The generalized Binet formula for $k$-bonacci numbers}, Elem. Math. \textbf{79} (2023).



    \bibitem{eS56} E. S. Selmer. \emph{On the irreducibility of certain trinomials}, Math. Scand. \textbf{4} (1956), 287--302.
\end{thebibliography}
\end{document}